\documentclass[12pt, oneside]{article}
\usepackage{amsmath,amsthm,amssymb,amsfonts}
\usepackage[shortlabels]{enumitem}
\usepackage{mathtools, mathrsfs, xargs}
\usepackage[dvipsnames]{xcolor}
\usepackage{hyperref}
\hypersetup{colorlinks, citecolor=Violet, linkcolor=Blue, urlcolor=Blue}
\newif\iflabels
\labelsfalse 
\iflabels 
  \usepackage[left=0.5in, right=2.5in, marginparwidth=2.2in]{geometry}
  \usepackage{showlabels}
\else
  \usepackage[left=1.00in, right=1.00in]{geometry}
\fi

\numberwithin{equation}{section}
\theoremstyle{plain}                
\newtheorem{theorem}{Theorem}[section]
\newtheorem{lemma}[theorem]{Lemma}
\newtheorem{proposition}[theorem]{Proposition}
\newtheorem{corollary}[theorem]{Corollary}

\theoremstyle{definition}           
\newtheorem{definition}[theorem]{Definition}
\newtheorem{example}[theorem]{Example}

\theoremstyle{remark}
\newtheorem{remark}[theorem]{Remark}

\newcommand{\define}[1]{{\textbf{#1}}}

\providecommand{\alias}{}
\renewcommand{\alias}[1]{\providecommand{#1}{}\renewcommand{#1}}

  \DeclarePairedDelimiter\ab{\langle}{\rangle} 
  \DeclarePairedDelimiter\abs{\lvert}{\rvert}   
  \DeclarePairedDelimiter\norm{\lVert}{\rVert}  

  \DeclarePairedDelimiterX\set[1]\{\}{ #1 }
  \DeclarePairedDelimiterX\sets[2]\{\}{ #1\,:\,#2 }

  \let\bPeexp\exp
  \let\exp\relax
  \DeclarePairedDelimiterXPP\exp[1]{\bPeexp}(){}{#1}

  \makeatletter
    \let\oldabs\abs \def\abs{\@ifstar{\oldabs}{\oldabs*}}
    \let\oldab\ab \def\ab{\@ifstar{\oldab}{\oldab*}}
    \let\oldnorm\norm \def\norm{\@ifstar{\oldnorm}{\oldnorm*}}
    \let\oldexp\exp \def\exp{\@ifstar{\oldexp}{\oldexp*}}
  \makeatother

\DeclareMathOperator\diam{diam}
\DeclareMathOperator*\esssup{esssup}

\makeatletter
  \newcommand{\opnorm}{\@ifstar\@opnorms\@opnorm}
  
  \newcommand{\@opnorm}[2][]{%
    \mathopen{#1|\mkern-1.5mu#1|\mkern-1.5mu#1|}
    #2
    \mathclose{#1|\mkern-1.5mu#1|\mkern-1.5mu#1|}
  }
\makeatother

\alias{\R}{{\mathbb R}}
\alias{\C}{{\mathbb C}}
\alias{\Z}{{\mathbb Z}}
\alias{\N}{{\mathbb N}}

\newcommand{\tf}[2]{\tfrac{#1}{#2}}
\newcommand{\oo}[1]{\tf{1}{#1}}
\newcommand{\tot}{\oo{2}} 

\newcommand{\pd}[2]{\frac{\partial #1}{\partial #2}}

\newcommand{\pD}[2]{\frac{\Delta #1}{\Delta #2}}

\newcommand{\ind}[1]{ 1_{{#1}}} 

\newcommand{\upn}[2]{#1^{(#2)}}

\newcommand{\efor}{\text{ for }}
\newcommand{\eforall}{\text{ for all }}
\newcommand{\eforeach}{\text{ for each }}
\newcommand{\eand}{\text{ and }}

\newcommand{\ewhere}{\text{ where }}

\newcommand{\vp}{\varphi}

\newcommand{\sA}{\mathcal{A}}
\newcommand{\sP}{\mathcal{P}}
\newcommand{\sS}{\mathcal{S}}
\newcommand{\sQ}{\mathcal{Q}}
\newcommand{\sF}{\mathcal{F}}

\newcommand{\bP}{\mathbb{P}}
\newcommand{\bF}{\mathbb{F}}
\newcommand{\bE}{\mathbb{E}}
\newcommand{\bQ}{\mathbb{Q}}
\newcommand{\bD}{\mathbb{D}}

\newcommand{\kz}{\boldsymbol{z}}
\newcommand{\kA}{\boldsymbol{A}}
\newcommand{\kB}{\boldsymbol{B}}

\newcommand{\kV}{\boldsymbol{V}}
\newcommand{\kZ}{\boldsymbol{Z}}

\newcommand{\tY}{U}

\newcommand{\lone}{L^1}
\newcommand{\ltwo}{L^2}
\newcommand{\lpee}{L^p}
\newcommand{\lque}{L^q}
\newcommand{\linf}{L^{\infty}}

\newcommand{\loi}{L^{1,\infty}}

\newcommand{\loq}{L^{1,q}}
\newcommand{\ltz}{L^{2,0}}
\newcommand{\lpz}{L^{p,0}}
\newcommand{\lto}{L^{2,1}}

\newcommand{\ltp}{L^{2,p}}
\newcommand{\ltq}{L^{2,q}}
\newcommand{\lqp}{L^{q,p}}
\newcommand{\lii}{L^{\infty, \infty}}

\newcommand{\sinf}{\sS^{\infty}}
\newcommand{\szer}{\sS^0}

\newcommand{\spee}{\sS^p}
\newcommand{\sque}{\sS^q}

\newcommand{\BMO}{\text{BMO}}
\newcommand{\bmo}{\text{bmo}}
\newcommand{\bmoh}{\bmo^{1/2}}

\newcommand{\mpee}{\mathcal{M}^p}

\newcommand{\rp}{(R_p)}

\newcommand{\rphr}{(\tilde{R}_p)}

\newcommand{\doi}{\bD^{1,\infty}}
\newcommand{\D}{\mathbb{D}^{1,2}}
\newcommand{\Doi}{{\mathbb D}^{1,\infty}}

\newcommand{\sibm}{\sinf\times\bmo}

\newcommand{\const}[1]{C=C(#1)}
\newcommand{\leqc}{\leq_C}

\newcommand{\qua}{\mathbf{Q}}

\newcommand{\tri}{\mathbf{T}}

\newcommand{\con}{\mathbf{C}}
\newcommand{\apb}{\mathbf{A}}

\newcommand{\mal}{\mathbf{M}}

\newcommand{\bsde}{\mathcal{B}}

\newcommand{\sam}{\set{a_m}}
\newcommand{\samp}{\set{a'_m}}

\newcommand{\sph}{\mathsf{S}^2}
\newcommand{\Saa}{S_{\alpha, \kA}}

\newcommand{\Sa}{S_{\kA}}
\newcommand{\ict}{\int_{\cdot}^T}

\newcommand{\bsdea}{\text{BSDE}\,(\kA)} 
\newcommand{\bsdeaa}{\text{BSDE}\,(\alpha, \kA)}

\newcommand{\ption}{(\tau_k)_{k=0}^m}
\newcommand{\itauk}{\ind{ [\tau_{k-1}, \tau_k] }}
\newcommand{\itauko}{\ind{ [\tau_{k-1}, \tau_k) }}
\newcommand{\itaup}{\ind{ [\tau,\tau'] }}
\newcommand{\itauT}{\ind{ [\tau_{m-1},T] }}
\newcommand{\Ptx}{\Phi^{\tau,\tau',\xi}}

\newcommand{\Yk}{\upn{Y}{k}}
\newcommand{\Yl}{\upn{Y}{l}}

\newcommand{\kZk}{\upn{\kZ}{k}}
\newcommand{\kZl}{\upn{\kZ}{l}}
\newcommand{\fk}{\upn{f}{k}}

\newcommand{\fl}{\upn{f}{l}}

\newcommand{\pik}{\upn{\pi}{k}}
\newcommand{\alphak}{\upn{\alpha}{k}}
\newcommand{\betak}{\upn{\beta}{k}}
\newcommand{\etak}{\upn{\eta}{k}}
\newcommand{\kAk}{\upn{\kA}{k}}
\newcommand{\xik}{\upn{\xi}{k}}
\newcommand{\xil}{\upn{\xi}{l}}

\begin{document}

\title{Existence and Uniqueness for non-Markovian Triangular Quadratic BSDEs}
\author{Joe Jackson and Gordan {\v Z}itkovi{\' c}} 
\date{}
\maketitle

\begin{abstract}
    We prove the existence and uniqueness of solutions to a class of quadratic BSDE systems which we call triangular quadratic. Our results generalize several existing results about diagonally quadratic BSDEs in the non-Markovian setting. As part of our analysis, we obtain new results about linear BSDEs with unbounded coefficients, which may be of independent interest. Through a non-uniqueness example, we answer a ``crucial open question" raised by Harter and Richou by showing that the stochastic exponential of an $n \times n$ matrix-valued $\BMO$ martingale need not satisfy a reverse H\"older inequality. 
\end{abstract}


\section{Introduction}

\subsection{Backward stochastic differential equations}

A backward stochastic differential equation (BSDE) is an expression of the form 

\begin{align} \label{introbsde}
    Y = \xi + \int_{\cdot}^T f( \cdot, Y, \kZ) dt - \int_{\cdot}^T \kZ d \kB. 
\end{align}
Here $\kB$ is a $d$-dimensional Brownian, $f = f(t, \omega, y, \kz) : [0,T] \times \Omega \to \R^n \times (\R^d)^n \to \R^n$ is a random field called the driver with various measurability and continuity constraints, and $\xi$ is an $n$-dimensional random vector called the terminal condition which is measurable with respect to $\sF_T$, where the filtration $ \bF = (\sF_t)_{0 \leq t \leq T}$ is the augmented filtration of $\kB$. A solution consists of two $\bF$-adapted processes $Y$ and $\kZ$, taking values in $\R^n$ and $(\R^d)^n$, respectively, satisfying \eqref{introbsde}. Our decision to view the co-domain of $\kZ$ as $(\R^d)^n$ rather than $\R^{n \times d}$, as well as our use of bold in \eqref{introbsde}, is due to notational considerations and will be explained in the notations and preliminaries section below. 

BSDEs are categorized largely based on the assumptions placed on the driver $f$. Roughly speaking, the equation \eqref{introbsde} is called
\begin{enumerate}
    \item linear if $f$ is linear in $y$ and $\kz$,
    \item Lipschitz if $f$ is Lipschitz in $y$ and $\kz$,
    \item quadratic if $f$ is Lipschitz in $y$ and depends quadratically on $\kz$, and
    \item Markovian if 
    \begin{align*}
        f(t, \omega, y, \kz) = \tilde{f}(t, X_t(\omega), y, \kz), \, \, 
        \xi = g(X_T),
    \end{align*}
    for some appropriate functions $\tilde{f}$, $g$, and diffusion $X$. 
\end{enumerate}
We also make a distinction between systems of BSDEs or multidimensional BSDEs ($n > 1$) and one-dimensional BSDEs ($n = 1$). 

BSDEs were first introduced by Bismut in \cite{Bis73}, who studied linear BSDEs in the context of stochastic control. In \cite{Pardoux-Peng}, Pardoux and Peng treated a general class of Lipschitz BSDEs, proving well-posedness of \eqref{introbsde} when $\xi \in L^2$. In \cite{Kobylanski}, Kobylanski provided an existence and uniqueness result for quadratic BSDEs in dimension one, under the assumption that $\xi$ is bounded. Quadratic systems have proved more challenging, and in fact a non-existence example in \cite{FreRei11} shows that a full generalization of Kobylanski's existence result to quadratic systems is impossible. In order to obtain existence results for quadratic BSDE systems, it is therefore necessary to make additional assumptions on the driver $f$ or the terminal condition $\xi$. One possibility is to impose smallness, as in \cite{Tevzadze}, where a fixed point argument is used to prove existence for quadratic systems when $\xi$ is small enough in $L^{\infty}$. 

Another possibility is to make additional structural assumptions on the driver. For example, in the Markovian setting, \cite{xing2018} establishes existence under a general structural constraint. For non-Markovian equations, some existence results have been obtained for ``diagonally quadratic" drivers (see \cite{HuTan16} and \cite{fanhutang}) and for drivers whose nonlinearity has a ``quadratic linear" form \footnote{We are following \cite{xing2018} by using the name quadratic linear to refer to the drivers studied in \cite{Nam19}, but the name is not actually used in \cite{Nam19}} (see \cite{Nam19}). 

The applications of quadratic BSDEs to stochastic optimal control, stochastic games, and financial economics (see, e.g. \cite{El-Karoui-Hamadene}, \cite{Cheridito-et.al}, \cite{Espinosa-Touzi}) and \cite{kardaras2017incomplete}), as well as to partial differential equations and even stochastic differential geometry (see \cite{Darling}) have been well-documented. Indeed, in \cite{pen99}, Peng lists existence for quadratic systems of BSDEs among the most important open questions in the field. This paper is motivated in part by the many applications of BSDEs, but also by the need for new probabilistic tools to study non-Markovian quadratic BSDEs.

\subsection{Our results}

\paragraph{Triangular Quadratic BSDEs.}

The main result of the paper, Theorem \ref{thm:main-tri}, is an existence and uniqueness result for equations whose drivers are triangular quadratic. This means, roughly speaking, that the driver $f$ is quadratic and its $i^{\text{th}}$ component depends sub-quadratically on the $j^{\text{th}}$ component of $\kz$, whenever $j > i$. In addition to this primary structural condition, we assume also that $f$ satisfies the (AB) condition from \cite{xing2018}, and that $f$ has some regularity in the sense of Malliavin calculus. Triangular quadratic drivers are generalizations of diagonally quadratic drivers, and in that sense our results generalize those of \cite{HuTan16} (though a strict comparison of the results is not possible because of the Malliavin regularity and (AB) assumptions). Our triangular quadratic drivers also generalize those considered in \cite{luo2020}, which have a triangular structure, but of a much more specific form.

To prove Theorem \ref{thm:main-tri}, we use the approach of \cite{Briand-Elie} for one-dimensional quadratic BSDEs. Namely, we first assume that $\xi$ has bounded Malliavin derivative and $f$ is smooth in $y$ and $\kz$, and we produce a sequence of Lipschitz drivers $f^k$ approximating $f$ with corresponding solutions $(Y^k, \kZ^k)$. We differentiate each of the approximate equations to get a linear BSDE, whose coefficients can be estimated in the space $\bmo$ (see the notations and preliminaries section below). Finally we apply estimates for linear BSDEs with $\bmo$ coefficients to conclude that $\sup_k \norm{\kZ^k}_{\linf} < \infty$, and thus $(Y^k,\kZ^k)$ solves the original equation when $k$ is sufficiently large. The success of this approach in dimension one relies on two key facts: 
\begin{enumerate}
    \item When $n = 1$, we can always guarantee that $\sup_k \big(\norm{Y^k}_{\sinf} + \norm{\kZ^k}_{\bmo}\big) < \infty$, i.e. the approximation scheme is bounded.
    \item There is a good theory for one-dimensional linear BSDEs with $\bmo$ coefficients, and in particular a-priori estimates for such equations are available. 
\end{enumerate}
Unfortunately, neither of these statements generalize to higher dimensions. Nevertheless, the condition (AB) does allow us to conclude that $\sup_k \big(\norm{Y^k}_{\sinf} + \norm{\kZ^k}_{\bmo}\big) < \infty$, and the triangular structure, together with the new results for linear BSDEs with $\bmo$ coefficients obtained in Section \ref{sec:linear}, provides the necessary estimates. 

This is not the first paper to use the approach of \cite{Briand-Elie} to study quadratic systems. The same general strategy was adopted by Harter and Richou in \cite{harter2019}, who use a similar approximation scheme but assume a-priori that $\sup_k \norm{\kZ^k}_{\bmo}$ is small enough, and then rely on the well-posedness of linear equations with small $\bmo$ coefficients. They proceed to check the a-priori smallness assumption in various cases, allowing them to recover results about BDSEs with small terminal condition or diagonally quadratic driver. Thus, while both the present paper and \cite{harter2019} are concerned with applying the strategy of \cite{Briand-Elie} to higher dimensions, \cite{harter2019} uses \textit{smallness} to overcome the difficulties presented by systems, while the present paper uses additional \textit{structural} assumptions.

We note that while our focus is specifically on the non-Markovian setting, our results provide new insights even in the Markovian case, since the $\linf$ estimates on $\kZ$ provided in Theorem \ref{thm:main-tri} amount in this case to an estimate on the Lipschitz constant of a Markovian solution. For example, our results show that under appropriate conditions, the locally H\"olderian Markovian solutions produced in \cite{xing2018} are actually Lipschitz. Interestingly, this additional regularity is achieved entirely through probabilistic arguments. 

To illustrate our main result, we prove in section \ref{sec:game} the existence of a Nash equilibrium in a simple two player game. The game contains a certain asymmetry between the two players, which leads to a triangular structure in the corresponding BSDE system. Thus the example sheds some light on what types of structures lead to triangular quadratic systems. The game is inspired by a semi-linear game treated in \cite{xing2018} and elsewhere. 

\paragraph{Linear BSDEs with bmo coefficients.}

In order to execute the strategy outlined above for triangular quadratic BSDEs, we develop some new results for linear BSDEs with $\bmo$ coefficients, which may be of independent interest. The results of \cite{Delbaen-Tang} and also \cite{harter2019} show that linear BSDEs with $\bmo$ coefficients are well-posed when their coefficients are small (or locally small, in the sense of sliceability). In Section \ref{sec:linear}, we show how smallness can be mixed with structural conditions on the coefficient matrix to get stronger results. For example, Corollary \ref{cor:sup-lin-exc} shows, roughly speaking, that we need only assume smallness above the diagonal of the coefficient matrix to get well-posedness. We also provide a non-uniquess example (Example \ref{exa:non-exist}) which demonstrates the necessity of either smallness or structural assumptions on the coefficients.

\paragraph{No reverse H\"older in higher dimensions.}

If $M$ is a $\BMO$ martingale, then its stochastic exponential $S = \mathcal{E}(M)$ is a uniformly integrable martingale which sastisfies the reverse H\"older inequality $\rp$ for some $p > 1$, i.e. the estimate 
\begin{align*}
    \bE_{\tau}[|S_T|^p] \leqc |S_{\tau}|^p
\end{align*}
holds for each stopping time $\tau$ with $0 \leq \tau \leq T$. In fact, this condition is essentially equivalent to membership in $\BMO$ (see Theorem 3.4 of \cite{Kazamaki}). The reverse H\"older inequality is an important tool which can be used to analyze linear BSDEs with $\bmo$ coefficients. When $M$ is instead an $n \times n$ matrix of $\BMO$ martingales, it is still possible to define the stochastic exponential $S$ of $M$, which is an $\R^{n \times n}$-valued local martingale. Likewise, one can generalize the reverse H\"older inequality $\rp$ to matrix-valued processes. Recognizing the potential applications to quadratic systems, Harter and Richou posed in Remark 3.5 of \cite{harter2019} the following ``crucial open question" : if $M$ is an $n \times n$ matrix of $\BMO$ martingales, does its stochastic exponential $S$ satisfy a reverse H\"older inequality? In Corollary \ref{cor:norp}, we use Example \ref{exa:non-exist} to answer this question in the negative. Our construction takes advantage of the non-uniqueness of martingales on manifolds, together with the connection between martingales on manifolds and BSDEs explained in \cite{Darling}. Indeed, the process $Y$ constructed in Example \ref{exa:non-exist} is (up to applying a coordinate chart) a non-constant martingale on the sphere with a constant terminal value.

\subsection{Structure of the paper}

In the remainder of the introduction, we fix notation and other conventions. Section \ref{sec:linear} contains our analysis of linear BSDEs with $\bmo$ coefficients, including both new well-posedness results and our non-uniqueness example. Section \ref{sec:linear} closes with a discussion of the reverse H\"older inequality. In Section \ref{sec:triangular}, we state and prove Theorem \ref{thm:main-tri}, our main existence and uniqueness result for triangular quadratic BSDEs.

\subsection{Notations and preliminaries}\label{sse:notation}

\paragraph{The probabilistic setup.} We fix a probability space $(\Omega, \sF,
  \bP)$ which hosts a $d$-dimensional Brownian motion $\kB$, and a deterministic
  time horizon $T < \infty$. The filtration $\bF = (\sF_t)_{0 \leq t \leq T}$ is
  the augmented filtration of $\kB$, and we use the shortcut $\bE_{\tau}[\cdot]$
  for the conditional expectation $\bE[\cdot| \sF_{\tau}]$.

\paragraph{Universal constants.} We fix a natural number $n$, which will be the
dimension of the unknown process $Y$. We emphasize here that $n$, $d$, and $T$
are considered fixed throughout the paper. A constant which only depends on $n$,
$d$, or $T$ is said to be universal. Depending on the context, constants that
depend on additional quantities  may also be called universal, but if such
additional dependencies exist, they will always be made clear. More precisely,
if a constant $C$ depends on $\norm{\gamma}_{\bmo}$ (in addition to $n$, $d$ or
$T$) we write $\const{\norm{\gamma}_{\bmo}}$. We use the notation $A \leqc B$
for $A \leq C B$ and follow Hardy's convention that the implied constant is
allowed to change from use to use.

\paragraph{Conventions for multi-dimensional processes.}
To curb the proliferation of indices, we use the following convention All our
processes take values in Euclidean spaces and each will be interpreted either as
a scalar, an $n$-dimensional vector (column by default), or an $n\times n$
matrix. Parting slightly from the norm, we allow the entries of these
linear-algebraic objects to take values either in $\R$ or in $\R^d$. To
distinguish between the two cases, we use the bold font for the $\R^d$-valued
case and the regular font for the $\R$-valued case. Matrix multiplication
retains the standard definition, with the proviso that either the $\R^d$-inner product or the scalar product of an $\R$-valued scalar and an $\R^d$-valued vector
be used in lieu of the scalar multiplication, as appropriate. This way, for
example if $\kA$ denotes a process with values in $(\R^d)^{n\times n}$ and $\kZ$
a process with values in $(\R^d)^n$, we interpret the former as an $n\times
n$-matrix-valued and the later as $n$-vector-valued, both with entries in
$\R^d$. Their product is a well-defined process $\gamma$ with values in $\R^n$:
\begin{align}\label{notation-1}
  \gamma = \kA \kZ \text{ means } \gamma^i = \sum_{k=1}^n \kA^i_k\cdot \kZk,
\end{align}
where $\cdot$ denotes the inner product on $\R^d$. With the above convention in
mind, we usually drop all the indices from notation. In the cases where they do
get included (mostly for clarity) we follow Einstein's convention of implicit
summation over repeating pairs of lower and upper indices. The indices of a
$\R^d$-valued vector $\kz$ are denoted by $\kz^{(1)},\dots, \kz^{(d)}$. 

\paragraph{Finite differences and derivatives}
Let $x_1, x_2 \in \R^m$ be fixed. For a function $F:\R^m\to \R$ and $1\leq j
  \leq m$, we define
  \[ \left(\pD{F}{x}\right)_j = \frac{F(x_2^1,\dots, x_2^{j-1}, x_2^j,
  x_1^{j+1},\dots, x_1^m) - F(x_2^1,\dots, x_2^{j-1}, x_1^j, x_1^{j+1},\dots,
  x_1^m)}{x_2^j - x_1^j},\] where the convention $0/0=0$ is used. We always
  interpret $\pD{F}{x}$ as a row vector. 
  
  If the components of $x$ split naturally into groups, as in the case
  $x=(y,\kz)\in \R^n \times ((\R^d)^n)$, we split the components of $\pD{F}{x}$
  accordingly. This way we ensure that the following ``total-differential"
  relationship holds when $x_i=(y_i, \kz_i)$, $i=1,2$:
  \begin{align*}
    F(y_2, \kz_2) - F(y_1, \kz_1) = \pD{F}{y}(y_2 - y_1) + \pD{F}{\kz}(\kz_2-\kz_1)
  \end{align*}
  We note that the product of $\pD{F}{\kz}$ and $(\kz_2-\kz_1)$ above needs to
  be interpreted as in \eqref{notation-1}, i.e., as a product of a row and a column
  vector with $\R^d$-valued components. 

  A similar notational philosophy is applied to derivatives, too. Given function
  $F:\R^n\times (\R^d)^n \to \R$, we set 
  \begin{align*}
    \pd{f}{y} = ( \pd{f}{y^j} )_j\quad \  \pd{f}{\kz^j} 
    =  (\pd{f}{\kz^j(1)}, \dots, \pd{f}{\kz^j(d)}),\quad
    \pd{f}{\kz} = (\pd{f}{\kz^j})_j.
  \end{align*}
  When applied to an $\R^n$-valued functions, both $\Delta$ and $\partial$ are
  applied componentwise without any notational changes. 

\paragraph{Integation conventions.}
When integrating, we often replace the upper or lower index of integration by
$\cdot$, indicating that we are dealing with a function/process of that index.
Moreover, we often drop the time-parameter of the integrand, and, in an act of
notation abuse, use $dt$ to denote integration with respect to Lebesgue measure.
This way, for example, $\ict \gamma\, dt$ denotes the process $t \mapsto
\int_t^T \gamma_u\, du$.

In the spirit of the previous paragraph, the Brownian motion $\kB$ is
interpreted as an $\R^d$-valued ``scalar'' process and, therefore, typeset in
bold. We use the notation $\int \kZ\, d\kB$ as the shortcut for a componentwise
sum of $d$ one-dimensional stochastic integrals.

\paragraph{Spaces of processes.}
Assuming that all Euclidean spaces are equipped with the standard Euclidean
norm, the definitions of the following spaces apply equally well to scalar,
vector of matrix-valued processes, with entries in $\R$ or $\R^d$:
\begin{itemize}
  \item For $1 \leq p \leq \infty$, $\lpee$ denotes the space of $p$-integrable
  random variables, vectors or matrices.
  \item For $1 \leq p \leq \infty$, $\spee$ denotes the space of all continuous
  processes $Y$ such that
  \begin{align*}
    \norm{Y}_{\sS^p} \coloneqq
    \norm{Y^*}_{\lpee} < \infty \ewhere Y^* = \sup_{0 \leq t \leq T} Y_t.
  \end{align*}
  We write $Y\in \szer$ if $Y$ is adapted an continuous.
  \item For $1 \leq p \leq \infty$, $\mpee$ is the set of all martingales in $\spee$. 
  \item For $1 \leq p ,q \leq \infty$, $\lqp$ denotes the space of progressive
  processes $\gamma$ such that
  \begin{align*}
    \norm{\gamma}_{\lqp}^2 \coloneqq
    \norm{\left(\int_0^T \abs{\kZ_t}^p dt\right)^{q/p}}_{L_p} < \infty.
  \end{align*}
  We write $\gamma \in \lpz$ if $\int_0^T \abs{\gamma}^p\, dt<\infty$, a.s.
  Processes in $\lqp$-spaces that agree $dt\otimes d\bP$-a.e, are identified,
  unless we explicitly  state otherwise. 
  \item $\BMO$ denotes the space of continuous martingales $M$ such that
  \begin{align*}
    \norm{M}_{\BMO} \coloneqq
    \esssup_{\tau} \norm{\bE_{\tau}[ \abs{M_T - M_{\tau} }^2 ]}_{L^{\infty}}^{\frac{1}{2}}
    < \infty,
  \end{align*}
  where the supremum is taken over all stopping times $0 \leq \tau \leq T$.
  \item $\bmo$ denotes the space of progressive processes $\gamma$ such that
  \begin{align*}
    \norm{\gamma}_{\bmo}^2 \coloneqq \sup_{\tau}
    \bE_{\tau}\left[\int_{\tau}^T \abs{\gamma}^2 ds\right]
    < \infty.
  \end{align*}
  \item $\bmo^{1/2}$ denotes the space of progressive processes $\beta$ such
  that
  \begin{align*}
    \norm{\gamma}_{\bmo^{1/2}} \coloneqq
    \sup_{\tau} \bE_{\tau}\left[\int_{\tau}^T \abs{\gamma}\, ds\right]
    < \infty, \text{ i.e., }
    \norm{\gamma}_{\bmo^{1/2}} = \norm{\sqrt{\abs{\gamma}}}^2_{\bmo}.
  \end{align*}
\end{itemize}

If necessary, we emphasize the co-domain of the space of processes under consideration, e.g. by writing $\bmo(\R^d)$ for the space of $\bmo$ processes taking values in $\R^d$. All of these spaces can be considered with respect to an equivalent probability
measure $\bQ$, which we notate in a natural way when necessary (e.g.,
$\lpee(\bQ)$).

\section{Linear BSDE with bmo coefficients}\label{sec:linear}

\subsection{Model estimates and a non-uniqueness example}
We start with two straightforward estimates for semimartingales which we think
of as solutions of linear BSDEs of martingale-representation type. These
estimates will serve as  model a-priori estimate for more general linear BSDEs.
\begin{proposition} \label{pro:model-1} There exists a universal constant $C$
  with the following property: suppose that $Y$ is a semimartingale with the
  decomposition
  \begin{align*}
    Y = Y_0 + \int_0^{\cdot} \beta_u\, du + \int_0^{\cdot} \kZ\, d\kB,
    \ewhere \int_0^{\cdot} \kZ\, d\kB \text{ is a martingale. }
  \end{align*}
  \begin{enumerate}
    \item  If   $Y_T\in \linf$ and $\beta \in \bmoh$ then $Y\in \sinf$,
    $\kZ\in\bmo$ and
    \begin{align*}
      \norm{Y}_{\sinf} + \norm{\kZ}_{\bmo} 
      \leqc \norm{Y_T}_{\linf} + \norm{\beta}_{\bmoh}.
    \end{align*}
    \item  If $Y_T\in \lque$ and $\beta \in \loq$ then $Y\in \sque$,
    $\kZ\in\ltq$ and
    \begin{align*}
      \norm{Y}_{\sque} + \norm{\kZ}_{\ltq} 
      \leqc \norm{Y_T}_{\lque} + \norm{\beta}_{\loq}.
    \end{align*}
  \end{enumerate}
\end{proposition}

If we add a linear $\kZ$-dependence into the drift term, the magnitude of
$(Y,\kZ)$ can still be estimated by the inputs, except that now we need to
measure the size of the \emph{whole path} of $Y$ instead of just its terminal
value:
\begin{proposition} \label{pro:model-2} Given $\kA\in\bmo$, there exists a
  universal constant $\const{\norm{\kA}_{\bmo}}$ with the following property:
  suppose that $Y$ is a semimartingale that admits a decomposition of the form
  \begin{align*}
    Y = Y_0 + \int_0^{\cdot} (\kA\cdot\kZ + \beta)\, dt + \int_0^{\cdot} \kZ\, d\kB,
  \end{align*}
  where $\int_0^{\cdot} \kZ\, d\kB$ is a martingale.
  \begin{enumerate}
    \item  If   $Y\in\sinf$ and $\beta \in \bmoh$ then  $\kZ\in\bmo$ and
    \begin{align*}
      \norm{\kZ}_{\bmo} \leqc \norm{Y}_{\sinf} + \norm{\beta}_{\bmoh}.
    \end{align*}
    \item  If $Y\in \sque$ and $\beta \in \loq$ then $\kZ\in\ltq$ and
    \begin{align*}
      \norm{\kZ}_{\ltq} \leqc \norm{Y}_{\sque} + \norm{\beta}_{\loq}.
    \end{align*}
  \end{enumerate}
\end{proposition}
\begin{proof}
  We omit the details as the proof follows the standard route, using the
  dynamics of $\abs{Y}^2$ and estimation via standard inequalities along the way
  (the BDG inequality is used in the $\lque$-case).
\end{proof}
\medskip

It is interesting to observe that the necessity of the inclusion of the entire
path of $Y$ on the right-hand side is not a defect of the method of proof. The
following example shows that in a clear way.  It exhibits a linear BSDE in
dimension $n=2$ with $\bmo$-coefficients which admits multiple solutions. As
such, if cannot satisfy the estimates of Proposition \ref{pro:model-1} since
they, in particular, imply uniqueness.

\begin{example} \label{exa:non-exist} We construct an $\R^{2\times 2}$-valued
  $\bmo$ process $A$ such that the following equation, where $n=2$ and $d=1$,
  admits a nontrivial solution:
  \begin{align} \label{non-exist-bsde}
    Y = \ict A Z \, dt - \int_{\cdot}^T Z\, dB
  \end{align}

  The construction will use some language from stochastic differential geometry.
  See \cite{Lee} for the definitions  of Riemannian metrics, connections, and
  Christoffel symbols, and see \cite{Emery} for the definition of martingales on
  manifolds. We will also use the connection between martingales on manifolds
  and BSDEs established by Darling in \cite{Darling}.

  First, we need an $\R$-valued martingale $M$ such that $\bP[M_T = \pi] =
    \bP[M_T = - \pi] = 1/2$. For example, we could set
  \begin{align*}
    M_t = \bE[\eta | \sF_t] \efor 0 \leq t \leq  T, \ewhere
    \eta = \begin{cases}
      \pi   & B_T \geq 0 \\
      - \pi & B_T < 0,
    \end{cases}
  \end{align*}
  By the martingale representation theorem, we have $M_t = \int_0^t U_s\, dB_s$
  for some square-integrable process $U$. In fact, since $\int U dB$ is a
  bounded martingale, $U \in \bmo(\R)$.

  Next, let $\sph$ denote the unit sphere in $\R^2$, and let $\phi : \sph
    \setminus \{(0,0,1)\} \to \R^2$ be the stereographic projection from the
    north pole:
  \begin{align*}
    \phi(x_1, x_2,y) = \Big( \frac{x_1}{1 - y}, \frac{x_2}{1-y}\Big)
    \efor (x_1,x_2,y) \in \sph\setminus\{(0,0,1)\}.
  \end{align*}
  Let $h$ denote the round Riemannian metric on $\sph$ (i.e. the metric induced
  by the inclusion into $\R^2$), and let $g = \phi_{*} h$ be the corresponding
  metric on $\R^2$. The Christoffel symbols of $g$ can be computed explicitly as
  \begin{align*}
    \Gamma_{ij}^k = \frac{-2}{1 + \abs{x}^2} \big(x_j\delta_{ik} +
    x_i \delta_{jk} - x_k \delta_{ij} \big), \, \, 1 \leq i,j,k \leq 2,
  \end{align*}
  where $x = (x_1,x_2) \in \R^2$ and $\delta_{\cdot, \cdot}$ denotes the
  Kronecker delta function.

  \medskip

  The path $t \mapsto (\cos(t), \sin(t))$ is a geodesic on $(\R^2, g)$, since it
  is mapped by $\phi^{-1}$ to a geodesic traveling along the equator of $\sph$
  at constant speed. Thus, the process
  \begin{align*}
    (\tY_t^1, \tY_t^2) = (\cos(M_t), \sin(M_t))
  \end{align*}
  is a martingale on $(\R^2, g)$ (see Proposition 4.32 of \cite{Emery} for
  details). Therefore by Lemma 2.2 of \cite{Darling}, there exists an
  $\R^2$-valued adapted process $Z$ such that $(\tY,  Z)$ satisfies the BSDE
  \begin{align*}
    \tY^1 = -1  - \frac{1}{2} \ict
    \left(\Gamma_{11}^1(\tY) ( Z^1)^2 +
    2\Gamma_{12}^1(\tY)  Z^1  Z^2 +
    \Gamma_{22}^1(\tY) (Z^2)^2\right)\, dt - \ict  Z^1\, dB \\
    \tY^2 =  - \frac{1}{2} \ict
    \left(\Gamma_{11}^2(\tY) (Z^1)^2 +
    2\Gamma_{12}^2(\tY)  Z^1  Z^2 +
    \Gamma_{22}^2(\tY) (Z^2)^2\right)\, dt - \ict  Z^2\, dB.
  \end{align*}
  We set $Y = \tY + (1,0)$ -  noting that $Y_T=0$, but that $Y$ itself is not
  trivial - and define the matrix processes $A$ by
  \renewcommand*{\arraystretch}{1.7}
  \begin{align*}
    A = - \frac{1}{2} \begin{pmatrix}
      \Gamma_{11}^1(Y)  Z^1 + 2 \Gamma_{12}^1(Y)  Z^2 &
      \Gamma_{22}^1(Y)  Z^2                             \\
      \Gamma_{11}^2(Y)  Z^1 + 2\Gamma_{12}^2(Y)  Z^2  &
      \Gamma_{22}^2(Y)  Z^2
    \end{pmatrix}
  \end{align*}
  so that $(Y, Z)$ solves \eqref{non-exist-bsde}.
  \smallskip

  To show that the coefficients of $A$ and $Z$ are, indeed, in $\bmo$, we resort
  to explicit computation:
  \begin{align*}
    dY_t^1 & = d\tY^1_t =
    - \sin(M_t) U_t\, dB_t - \frac{1}{2} \cos(M_t) \abs{U_t}^2\, dt, \\
    dY_t^2 & = d\tY^2_t=
    \hspace{0.9em}\cos(M_t) U_t \,dB_t - \frac{1}{2} \sin(M_t) \abs{U_t}^2\, dt
  \end{align*}
  and so $ Z_t^1 = -\sin(M_t) U_t$, $ Z_t^2 = \cos(M_t) U_t$. Since $U \in
    \bmo(\R)$, it follows that $ Z_t^1,  Z_t^2 \in \bmo(\R)$, and, hence, that
    $A \in \bmo$, too.
\end{example}

The phenomenon brought forward in Example \ref{exa:non-exist} is exclusively a
multidimensional one. The change-of-measure techniques available in dimension
$1$ lead to the following well-known results:

\begin{proposition}\label{pro:1d-lin-bsde} Assume that $n=1$. Given $\kA\in \bmo$
  we consider the scalar BSDE
  \begin{align}
    \label{1d-lin-bsde}
    Y = \xi+ \ict \Big( \kA \kZ + \beta\Big)\, dt - \ict \kZ\, d\kB.
  \end{align}
  \begin{enumerate}
    \item If $\xi \in \linf(\R)$, then there exists a unique solution $(Y,\kZ)
    \in \sibm$ to \eqref{1d-lin-bsde} and it satisfies
    \begin{align*}
      \norm{Y}_{\sinf} + \norm{\kZ}_{\bmo} \leqc
      \norm{\xi}_{\linf} + \norm{\beta}_{\bmoh},\ \const{\norm{\kA}_{\bmo}}
    \end{align*}
    \item There exists a universal constant $q^* = q^*(\norm{\kA}_{\bmo})\geq 1$
    with the following property: for any $q>q^*$ and any $\xi \in \lque$, there
    exists a unique solution $(Y,\kZ) \in \sque \times \ltq$ to
    \eqref{1d-lin-bsde} and it satisfies
    \begin{align*}
      \norm{Y}_{\sque} +\norm{\kZ}_{\ltq} \leqc \norm{\xi}_{\lque} +
      \norm{\beta}_{\loq}, \ \const{q, \norm{\kA}_{\bmo}}.
    \end{align*}
  \end{enumerate}
\end{proposition}

\begin{proof}
The proof uses a standard change-of-measure argument, and we provide only a sketch. The idea is to rewrite \eqref{1d-lin-bsde} as
\begin{align*}
    Y = \xi + \int_{\cdot}^T \beta \, dt - \int_t^T \kZ d \kB^{\kA}, \, \, \kB^{\kA} = \kB - \int \kA dt,  
\end{align*}
and notice that $\kB^{\kA}$ is a martingale under the measure $\bQ$, where $\frac{d\bQ}{d\bP} = \mathcal{E}(\int \kA d\kB)$. Standard facts about exponentials of $\BMO$ martingales can then be applied to give the desired results. We refer to the proofs of Propositions 2.2 and 2.3 of \cite{Briand-Elie} for more details. 
\end{proof}

\subsection{A-priori estimates under triangularity and sliceability}
The message of the previous section is that that an additional assumption on the
coefficients - beyond membership in $\bmo$ - will be necessary for good a-priori
estimates in higher dimensions. The goal of this section is to present such an
additional assumption.

A general linear BSDE studied in this section will take the following form:
\begin{align} \label{lin-bsde}
  Y = \xi + \ict \Big(\alpha Y + \kA \kZ + \beta\Big)\, dt - \ict \kZ\, d\kB,
\end{align}
where $Y$ is a $\R^n$-valued and $\kZ$ is $(\R^d)^n$-valued. The coefficients
$\alpha, \kA$ and $\beta$ are $\R^{n\times n}$-, $(\R^d)^{n\times n}$- and
$\R^n$ valued, respectively, while $\xi$ is an $\R^n$-valued random vector. If
one insisted on including all $n$-dimensional indices, \eqref{lin-bsde} would be
written as:
\begin{align*}
  Y^i = \xi^i + \ict \Big(\alpha^i_j Y^j
  + \kA^i_j \cdot \kZ^j + \beta^i\Big)\, dt - \ict \kZ^i\cdot d\kB,\ 1\leq i \leq n.
\end{align*}

In general, a pair $(Y,\kZ)\in \szer\times\ltz$ is said to be a
\define{solution} to \eqref{lin-bsde} if $\int_{0}^{\cdot}\kZ\, d\kB$ is a
martingale and \eqref{lin-bsde} holds pathwise, a.s. When $(Y,\kZ)$ admits more
regularity, e.g., when $Y\in\sinf$ and $\kZ\in\bmo$, we say that $(Y,\kZ)$ is an
\define{$\sibm$-solution}. 

\medskip

We begin our analysis by abstracting the key property, and then providing
sufficient conditions in terms of two qualitatively different requirements. 

\begin{definition}
  For $(\alpha, \kA) \in \bmoh \times \bmo$, we say that the \define{BSDE($\alpha,\kA$) is well-posed} if
  for each pair $(\xi,\beta)\in\linf\times\bmoh$
  \begin{itemize}
    \item the BSDE \eqref{lin-bsde} admits a unique $\sibm$-solution  $(Y,\kZ)$
    and
    \item there exists a universal constant $\const{\alpha, \kA}$ such that 
    \begin{align*}
      \norm{Y}_{\sinf} + \norm{\kZ}_{\bmo} \leqc \norm{\xi}_{\linf} + \norm{\beta}_{\bmoh}.
    \end{align*}
  \end{itemize} 
  We say that \define{$\bsdea$ is well-posed} if BSDE($0, \kA$) is well-posed.
\end{definition}

If $\bsdeaa$ is well-posed, then there is a bounded solution operator
\begin{align*}
  \Saa: \linf\times\bmoh \to \sibm,\ \Saa (\xi,\beta) = (Y,\kZ)
\end{align*}
whose operator norm is  denoted by $\opnorm{\Saa}$. For $\alpha=0$ we write
$\Sa$ instead of $S_{0,\kA}$. When BSDE($\alpha,\kA$) is not well-posed, we
set $\opnorm{\Saa}=+\infty$. 

Proposition \ref{pro:1d-lin-bsde} states that when $n = 1$, $\bsdea$ is well-posed for any $\kA\in\bmo$. Example \ref{exa:non-exist} above, however, implies that this is
no longer the case in higher dimensions. The following proposition gives a
simple, but far reaching, criterion for well-posedness in any dimension. Since it
will be used in the proof  and beyond, we note that the conditional
Cauchy-Schwarz inequality implies that for all algebraically compatible $\gamma,
\rho \in \bmo$ we have
\begin{align}
  \label{bmo-holder}
  \norm{\gamma \rho}_{\bmoh} \leqc
  \norm{\gamma}_{\bmo} \norm{\rho}_{\bmo}.
\end{align}
\begin{proposition}\label{pro:lem:tri-dif-exc} Suppose that $\kA$ is lower
  triangular, i.e., that $\kA^i_j = \boldsymbol{0}$ for all $j>i$. Then $\bsdea$ is
  well-posed and 
  \begin{align*}
    \opnorm{\Sa} \leq C, \, \, C = C(\norm{\kA}_{\bmo}) .
  \end{align*}
\end{proposition}
\begin{proof} We pick a lower-triangular $\kA$ and set $\alpha=0$ so that the
  first row of \eqref{lin-bsde} reads
  \begin{align} \label{first-row}
    Y^1 = \xi^1 + \ict \left(\kA^{1}_{1} \cdot \kZ^1 + \beta^1\right)\, dt
    - \ict \kZ^1\cdot d\kB.
  \end{align}
  By Proposition \ref{pro:1d-lin-bsde}, the scalar BSDE \eqref{first-row} has a
  unique solution $(Y^1, \kZ^1) \in \sinf(\R) \times \bmo(\R^d)$ and it
  satisfies
  \begin{align} \label{first-row-est}
    \norm{Y^1}_{\sinf} + \norm{\kZ^1}_{\bmo} \leqc \norm{\xi^1}_{\linf}
    + \norm{\beta^1}_{\bmoh}
    \leq \norm{\xi}_{\linf} + \norm{\beta}_{\bmoh}, \, \, C = C(||\kA||_{\bmo}). 
  \end{align}
  Now that $(Y^1, \kZ_1^1)$ is uniquely determined by the first line of the
  equation, the second line reads
  \begin{equation} \label{second-row}
    \begin{split}
      Y^2
      & = \xi^2 + \ict \Big( \kA^2_2\cdot \kZ^2 +
      \kA^2_1\cdot  \kZ^1 + \beta^2 \Big)\, ds - \ict \kZ^2\, d\kB \\
      & = \xi^2 + \ict \Big( \kA^2_2\cdot \kZ^2 + \hat{\beta^2}
      \Big)\, ds - \ict \kZ^2\, d\kB
    \end{split}
  \end{equation}
  where  $\hat{\beta^2} = \kA^2_1\cdot  \kZ^1 + \beta^2 $.  The inequality
  \eqref{bmo-holder} implies that
  \begin{equation*} 
    \begin{split}
      \norm*{\hat{\beta}^2}_{\bmoh} &\leq \norm{\beta^2}_{\bmoh}
      +\norm{\kA_1^2\cdot \kZ^1}_{\bmoh} \\
      &\leq \norm{\beta^2}_{\bmoh} + \norm{\kA_1^2}_{\bmo} \norm{\kZ^1}_{\bmo}
      \leqc \norm{\beta}_{\bmoh} + \norm{\xi}_{\linf},
    \end{split}
  \end{equation*}
  where $\const{\norm{\kA}_{\bmo}}$ and the last inequality follows from
  \eqref{first-row-est}. Thus, by Proposition \ref{pro:1d-lin-bsde}, there is a
  unique solution $(Y^2, \kZ^2) \in \sibm$  to \eqref{second-row} satisfying
  \begin{align*}
    \norm{Y^2}_{\sinf} + \norm{\kZ^2}_{\bmo} \leqc \norm{\xi}_{\linf} +
    \norm{\beta}_{\bmoh},\ \const{\norm{\kA}_{\bmo}}.
  \end{align*}
  Continuing in this manner, we produce a solution $(Y,\kZ) \in \sibm$ to
  \eqref{lin-bsde} satisfying the stated bounds. Thanks to \ref{pro:1d-lin-bsde},
  this solution is also unique, so we conclude that $\bsdea$ is, indeed, well-posed.
\end{proof}
Sufficient sliceability - as introduced in the following definition - will play
the role of smallness in our main well-posedness criterion below.
\begin{definition}[Sliceability] \
  \begin{enumerate}
    \item A \define{random partition} of $[0,T]$ is a collection $\ption$  of
    stopping times such that $0=\tau_0 \leq \tau_1 \leq \dots \leq \tau_m=T$.
    The set of all random partitions is denoted by $\sP$.

    \item   For $\kA\in\bmo$, the \define{index of sliceability for $\kA$} is
    the function $N_{\kA}:(0,\infty)\to \N\cup\set{\infty}$ defined as follows.
    For $\delta>0$, $N_{\kA}(\delta)$ is the smallest natural number $m$ such
    that there exists a random partition $\ption \in \sP$ such that
    \begin{align}
      \label{slic}
      \norm{\kA \itauk}_{\bmo} \leq \delta \eforall 1\leq k \leq m.
    \end{align}
    If no such $m$ exists, we set $N_{\kA}(\delta)=\infty$.

    \item A bmo-process $\kA$ is said to be \define{$\delta$-sliceable} if
    $N_{\kA}(\delta)<\infty$ and \define{sliceable} if it is $\delta$-sliceable
    for each $\delta>0$. 

    \item A family $\sA \subseteq \bmo$ is said to be \define{uniformly
    sliceable} if
    \begin{align*}
      \sup_{\kA\in \sA} N_{\kA}(\delta)<\infty \eforall \delta>0.
    \end{align*}
  \end{enumerate}
  Sliceability and the related notions given above are defined for the space
  $\bmoh$ in the same way.
\end{definition}
\begin{remark} \label{rem:slic} It is well known that not every $\bmo$ process
  is sliceable, even in dimension $1$ (see \cite[Example 3.1, p. 349]{Sch96}).
  One of the simplest ways to ascertain sliceability  of $\gamma\in\bmo$ is to
  show that it is bounded or that $\abs{\gamma}^p \in \bmo$ for some $p>1$. For
  a bounded $\gamma$, we further have the following simple estimate which will
  be useful in the sequel:
  \begin{align}\label{bd-slic}
    N_{\gamma}(\delta) \leq 1+\norm{\gamma}_{\lii}/\delta 
    \leq 1+\norm{\gamma}_{\bmo}/\delta
  \end{align}
  More generally, it is enough to construct a nondecreasing function
  $\vp:[0,\infty) \to [0,\infty)$ such that $\vp$ is convex and
  $\lim_{x\to\infty} \tfrac{\vp(x)}{x} = \infty$ with the property that
  $\sqrt{\vp(\abs{\gamma}^2)} \in \bmo$. Indeed, for $0\leq a < b\leq T$, by the
  conditional Jensen's inequality we have
  \begin{align*}
    \vp\left( \oo{b - a}
    \bE_{a} \left[   \int_{a}^{b} \abs{\gamma}^2\, dt \right] \right)
     & \leq
    \bE_{s} \left[  \oo{b-a} \int_{a}^{b} \vp( \abs{\gamma}^2)\, dt\right] 
      \leq \oo{b - a} \norm{\sqrt{\vp(\abs{\gamma}^2)}}^2_{\bmo}.
  \end{align*}
  Therefore, $\bE_{a} \int_{a}^{b} \abs{\gamma}^2\, dt$ can be made uniformly
  arbitrarily small by making $b-a$ small enough. Hence, given any $\delta>0$, a
  fine-enough deterministic partition can be used show that $\gamma$ is
  $\delta$-sliceable. In particular, $N_{\gamma}(\delta) \leq C$, where
  $\const{\delta, \norm{\gamma}_{\bmo}, \vp}$. 
\end{remark}

 Our next result, Theorem \ref{thm:exc-neigh} shows that if
 BSDE($\alpha,\kA$) is well-posed, its ``neighborhood" - measured by sliceability - is
 also well-posed.

\begin{theorem}\label{thm:exc-neigh} Suppose that $(\alpha,\kA)\in\bmo$, and $\bsdeaa$ is well-posed. Then there exists a constant $\delta = \delta(\opnorm{\Sa})>0$ such
  that $\text{BSDE}\,(\alpha+\Delta\alpha, \kA+\Delta \kA)$ is well-posed as soon as both
  $\Delta \alpha$ and $\Delta \kA$ are   $\delta$-sliceable. Moreover
  \begin{align}\label{norm-sakd} 
    \opnorm{S_{\alpha+\Delta \alpha, \kA+\Delta \kA}} \leqc 
     N_{\Delta \alpha}(\delta)+N_{\Delta \kA}(\delta).
  \end{align}
\end{theorem}
\begin{proof} We fix $\Delta \alpha\in\bmoh$ and $\Delta \kA\in\bmo$. For
  $(R,\kV)\in \sibm$, stopping times $0\leq \tau \leq \tau'\leq T$ and a random
  variable $\xi \in \linf(\sF_{\tau'})$ we define
  \begin{align*}
    \Ptx (R,\kV) = \Saa\Big(\xi, (\Delta \alpha\, R 
    + \Delta\kA\, \kV + \beta) \itaup \Big).
  \end{align*}
  noting that $\Ptx(R,\kZ) = (\xi,0)$ on $(\tau',T]$.

  To estimate the contractivity of $\Ptx$ we assume that $\Ptx(R^k,\kV^k) =
  (\Yk, \kZk)$, $k=1,2$ and set $\Delta Y = Y^1 - Y^2$, $\Delta R = R^2-R^1$,
  $\Delta \kZ = \kZ^1 - \kZ^2$, and $\Delta \kV = \kV^1 - \kV^2$, so that
  $(\Delta Y, \Delta \kZ)$ solves the  BSDE
  \begin{align*}
    (\Delta Y, \Delta \kZ) = \Saa\Big(0, (\Delta\kA \Delta \kZ 
    + \Delta \alpha \Delta R) \itaup\Big).
  \end{align*}
  Therefore, 
  \begin{align*}
    \norm{\Delta Y}_{\sinf} + \norm{\Delta \kZ}_{\bmo}
     \leq \hspace{-8em}&\hspace{8em} \opnorm{\Saa} 
     \norm*{\Delta \kA\cdot \Delta \kV \itaup + 
     \Delta\alpha \Delta R \itaup}_{\bmoh} \\
     & \leqc \opnorm{\Saa}\norm*{\Delta \kA \itaup}_{\bmo}
      \norm{\Delta \kV}_{\bmo}
    + \norm{\Delta\alpha \itaup }_{\bmoh} \norm{\Delta R}_{\sinf}\\
     & \leqc \opnorm{\Saa}\big( \norm*{\Delta \kA\itaup}_{\bmo} + 
     \norm*{\Delta\alpha_{\itaup}}_{\bmoh} \big)
    \big(\norm{\Delta R}_{\sinf} + \norm{\Delta \kV}_{\bmo}\big).
  \end{align*}
  It follows that there exists a universal constant $\delta =
  \delta(\opnorm{\Sa})$ such that $\Ptx$ is $\tot$-Lipschitz on $\sibm$ as soon
  as $\max(\norm{\Delta A\itaup}_{\bmo}, \norm{\Delta\alpha\itaup}_{\bmoh})\leq
  \delta$. Having fixed such a constant $\delta$  we assume that both
  $\Delta\alpha$ and $\Delta \kA$ are $\delta$-sliceable, as witnessed by the
  partition $\ption$. We note that a single partition of size at most
  $N_{\kA}(\delta) + N_{\alpha}(\delta)$ can always be chosen to work for both
  processes.

  Under these sliceability conditions, $\Phi^{\tau_{m-1},T,\xi}$ is
  $\tot$-Lipschitz on $\sinf\times \bmo$, and, as such, admits a fixed point
  $(\upn{Y}{m}, \upn{\kZ}{m})$. By the definition of $\Phi^{\tau_{m-1},T,\xi}$,
  the pair $(\upn{Y}{m}, \upn{\kZ}{m})$ solves the equation
  \begin{align} \label{delta-bsde}
    Y = \xi+ \ict \Big((\alpha+\Delta \alpha) Y + (\kA+\Delta \kA) 
    \cdot  \kZ+\beta\Big)\, dt
    - \ict  \kZ\, d\kB.
  \end{align}
  on $[\tau_{m-1}, T]$.

  Next, let $(\upn{Y}{m-1}, \upn{\kZ}{m-1})$ denote unique fixed point of
  $\Phi^{\tau_{m-2}, \tau_{m-1}, \xi^{m-1}}$ where we use $\xi^{m-1} =
  \upn{Y}{m}_{\tau_{m-1}}$ as the terminal condition.  As above, $(\upn{Y}{m-1},
  \upn{\kZ}{m-1})$  solves \eqref{delta-bsde} on $[\tau_{m-2},\tau_{m-1}]$ and
  $\upn{Y}{m-1}$ agrees with $\upn{Y}{m}$ at $\tau_{m-1}$. Continuing in this
  manner, we construct the following solution to \eqref{delta-bsde}
  \begin{align*}
    Y = \sum_{k=1}^m  \upn{Y}{k}\itauko + \xi \ind{\set{T}}, 
    \quad  \kZ = \sum_{k=1}^m  \upn{\kZ}{k} \itauko.
  \end{align*}

  \medskip

  The universal bounds are established step by step, as well. We pick a solution
  $(Y,\kZ)\in \sibm$. Suitably modified on $[0,\tau_{m-1})$ it is a fixed point
  of the map $\Phi^{\tau_{m-1},T,\xi}$. Since $\Phi^{\tau_{m-1},T,\xi}$  is
  $\tot$-Lipschitz, the distance between the fixed point of
  $\Phi^{\tau_{m-1},T,\xi}$ and its value $(Y^0,
  \kZ^0)=\Phi^{\tau_{m-1},T,\xi}(0,0)$ is at most
  $\norm{Y^0}_{\sinf}+\norm{\kZ^0}_{\bmo}$. Therefore,
  \begin{align*}
    \norm{Y \itauT}_{\sinf} + \norm{\kZ \itauT}_{\bmo} 
    &\leq 2 \Big( \norm{Y^0}_{\sinf}+\norm{\kZ^0}_{\bmo}\Big)\\
    &\leqc \norm{\xi}_{\linf} + \norm{\beta}_{\bmoh}.
  \end{align*}
  We continue in the same manner, noticing that $(Y,\kZ)$ can be modified on
    $[0,T] \setminus [\tau_{m-2}, \tau_{m-1}]$ to become a fixed point of
    $\Phi^{\tau_{m-1}, \tau_m, Y_{\tau_{m-1}}}$.  Same as above, we conclude
    that
  \begin{align*}
    \norm{Y \ind{[\tau_{m-2},\tau_{m-1}]}}_{\sinf} + 
    \norm{\kZ \ind{[\tau_{m-2},\tau_{m-1}]}}_{\bmo}
    \leqc \norm{Y_{\tau_{m-1}}}_{\linf} + \norm{\beta}_{\bmoh}.
  \end{align*}
  Continuing this way, we obtain \eqref{norm-sakd} - which, in turn, implies
  uniqueness - after $m \leq N_{\Delta \kA}(\delta) + N_{\alpha}(\delta)$ steps.
\end{proof}
Proposition \ref{pro:lem:tri-dif-exc} and various sufficient conditions for
sliceability from Remark \ref{rem:slic} lead to the following consequences of
Theorem \ref{thm:exc-neigh} above:
\begin{corollary} \label{cor:one-slic-exc} Let $\kA\in\bmo$ and
$\alpha\in\bmoh$. 
  \begin{enumerate}
    \item If $\kA$ and $\alpha$ are sliceable then $\bsdeaa$ is well-posed.
    \item If $\bsdea$ is well-posed and $\alpha$ is sliceable, $\bsdeaa$ is well-posed. 
    \item If $\bsdea$ is well-posed and $\alpha$ is bounded then 
    \begin{align*}
      \opnorm{\Saa} \leq C,\ \const{\opnorm{\Sa}, \norm{\alpha}_{\lii}}.
    \end{align*}
  \end{enumerate}
\end{corollary}
The following corollary plays the key role in our treatment of nonlinear BSDEs
with triangular quadratic drivers in the next section. 
\begin{corollary}\label{cor:sup-lin-exc} Suppose that $\kA\in\bmo$, $\alpha \in \lii$ and the
  convex superlinear function $\vp:[0,\infty)\to [0,\infty)$ are such that
  \begin{align*}
    \vp(\abs{\kA^i_j}^2)^{1/2} \in \bmo \efor j>i.
  \end{align*}
  Then $\bsdeaa$ is well-posed and $\opnorm{S_{\alpha,\kA}}\leq C$, where
  $\const{\norm{\alpha}_{\lii}, \norm{\kA}_{\bmo}, \vp}$. 
\end{corollary}

The results of this section all have $L^q$ analogues, proven in essentially the same way, with part 2 of Proposition \ref{pro:model-2} replacing part 1. Rather than develop the whole theory in parallel, we have decided to state only what we need for application to quadratic BSDEs. 

\begin{proposition} \label{pro:lqexc}
  Suppose that $\kA \in \bmo$, $\alpha \in \lii$, and the
  convex superlinear function $\vp:[0,\infty)\to [0,\infty)$ are such that
  \begin{align*}
    \vp(\abs{\kA^i_j}^2)^{1/2} \in \bmo \efor j>i.
  \end{align*}
  Then there exists $q^* = q^*(||A||_{\bmo}, \phi)$, such that for each $q > q^*$ and $\xi \in \lque$, there is a unique solution $(Y,Z) \in \sque \times \ltq$ to the BSDE
  \begin{align*}
      Y = \xi + \int_{\cdot}^T \big(\alpha Y + \kA \kZ + \beta\big) dt - \int_{\cdot}^T \kZ d\kB,
  \end{align*}
  and we have the estimate 
  \begin{align*}
      ||Y||_{\sque} + ||\kZ||_{\ltq} \leq_C ||\xi||_{\lque} + ||\beta||_{\loq}, \, \, \const{q,\norm{\kA}_{\bmo}, \norm{\alpha}_{L^{\infty}}, \phi}.
  \end{align*}
  
\end{proposition}

\subsection{The reverse H\"older inequality}
If $M$ is a real-valued $\BMO$ martingale, then the stochastic exponential $\mathcal{E}(M)$ is a true martingale which satisfies a Reverse H\"older inequality, a fact which will be stated precisely below. In this section, we use Example \ref{exa:non-exist} to show that an analogous statement does not hold for matrices of $\BMO$ martingales, which answers a ``crucial open question" raised in Remark 3.5 of \cite{harter2019}. Throughout the section, $\kA$ refers to a $\bmo$ process taking values in $(\R^d)^{n \times n}$. The following two definitions are adapted from Definition 3.1 of \cite{Delbaen-Tang}. 

\begin{definition} \label{sdef}
Let $M$ be a martingale taking values in $\R^{n \times n}$. We define the \textbf{stochastic exponential of $M$}, written $\mathcal{E}(M)$, to be the unique solution $S$ to the matrix SDE 
\begin{align} \label{expdef}
    dS = S dM, \, \, S_0 = I_{n \times n},
\end{align}
where $I_{n \times n}$ denotes the $n \times n$ identity matrix. 
\end{definition}
That \eqref{expdef} has a unique solution follows from \cite{Pro04} Theorem 3.7. Notice that $\int \kA d \kB$ is an $\R^{n \times n}$-valued local martingale. For the rest of the section, we write $S$ for the stochastic exponential of $\int \kA d\kB$, suppressing the dependence on $A$. That is, $S$ refers to the solution of the matrix SDE\begin{align*}
    dS = S \kA d\kB, \, \, S_0 = I_{n \times n}.
\end{align*}

\begin{remark} \
\begin{enumerate}
    \item Our notation for the stochastic exponential differs from that of \cite{Delbaen-Tang} by a transpose, but agrees with that of \cite{harter2019}. 
    \item Note that if $n = 1$, then $\mathcal{E}(M) = \exp{M - \frac{1}{2} \langle M \rangle}$, i.e. the matrix-valued stochastic exponential agrees with the usual stochastic exponential in dimension one. 
\end{enumerate}
\end{remark}

\begin{definition} \label{rp}
Given $1 \leq p < \infty$, we say that $S$ satisfies the reverse H\"older inequality $\rp$ if $S$ is a true martingale and the estimate
\begin{align}  \label{rhdef}
\bE_{\tau}[|S_T|^p] \leqc |S_{\tau}|^p
\end{align}
holds for some matrix norm $| \cdot |$ and all stopping times $\sigma$ with $0 \leq \sigma \leq T$.
\end{definition}

\begin{remark} \
\begin{enumerate}
    \item We have included the condition that $S$ is a true martingale in order to simplify various statements in the remainder of this sub-section. 
    \item Testing $\eqref{rhdef}$ with $\tau = 0$ reveals that if $S$ satisfies $\rp$, then $S \in \mpee$, and in particular $S$ is uniformly integrable. 
\end{enumerate}
\end{remark}

With this notation in place, we can recall the following basic fact about $\BMO$ martingales in dimension one, which combines Theorems 2.3 and 3.1 of \cite{Kazamaki}. 

\begin{theorem} \label{thm:1drh}
 If $n = 1$ and $\kA \in \bmo$, then $S$ satisfies $\rp$ for some $p = p(\norm{\kA}_{\bmo}) > 1$. 
\end{theorem}

Recognizing the applications of the reverse H\"older inequality to linear BSDEs with $\bmo$ coefficients, and hence to quadratic BSDE systems, the authors of \cite{harter2019} raised the following ``crucial open question" in Remark 3.5: when $n > 1$, does an analogue of Theorem \ref{thm:1drh} hold? The answer is yes when $\kA$ is sliceable, see Remark 3.2 of \cite{Delbaen-Tang}. In the remainder of this sub-section, we answer this question in the negative. First, we recall that solutions to the equation

\begin{align} \label{lin-bsde-2}
Y = \xi + \int_{\cdot}^T \kA \kZ dt - \int_{\cdot}^T \kZ d \kB
\end{align}
can be represented using $S$, a well-known principle (see \cite{Delbaen-Tang} and \cite{harter2019}), which we prove for the sake of completeness. First, we show that $S$ is invertible.

\begin{lemma} \label{lem:sinv}
If $\kA \in \bmo$, then for each $t \in [0,T]$, $S_t$ is invertible a.s. 
\end{lemma}

\begin{proof}
Define $X$ by the matrix SDE 
\begin{align*}
    dX = \kA^2 X dt - \kA X d\kB, \, \, X_0 = I_{n \times n},
\end{align*}
which has a unique solution by \cite{Pro04} Theorem 3.7. An application of the product rule reveals that 
\begin{align*}
    d \big(XS) = d \big(SX) = 0, 
\end{align*}
which implies that for each $t \in [0,T]$, we have 
\begin{align*}
    I_{n \times n} = X_tS_t = S_t X_t. 
\end{align*}
This completes the proof. 
\end{proof}

\begin{proposition} \label{pro:yrep}
  Suppose that $S$ is a true martingale. Then if $(Y,\kZ) \in \sinf \times \bmo$ solves \eqref{lin-bsde-2}, then for each $t$, we have
  \begin{align*}
      Y_t = S_t^{-1}\bE_t[S_T \xi]. 
  \end{align*}
  In particular, there is at most one solution to \eqref{lin-bsde-2} in $\sinf \times \bmo$. 
\end{proposition}

\begin{proof}
We compute 
\begin{align*}
    d (SY) = (dS) Y + S (dY) + (dS)(dY) \\ = (S \kA d\kB) Y + S\big(-\kA \kZ dt + \kZ d\kB\big) + (S \kA d\kB)(\kZ d\kB) \\
    = \big(S \kA Y + S \kZ \big) d \kB. 
\end{align*}
Since $S$ is a true martingale and $Y$ is bounded, $SY$ is of class (DL), hence a true martingale. In particular, for each $t$ we have 
\begin{align*}
    S_tY_t = \bE_t[S_T Y_T] = \bE_t[S_T \xi], 
\end{align*}
and so by the invertibility of $S_t$,
\begin{align*}
    Y_t = S_t^{-1} \bE_t[S_T \xi]. 
\end{align*}
\end{proof}

\begin{corollary} \label{cor:norp}
Let $A$ be the $\R^{2 \times 2}$-valued $\bmo$ process constructed in Example \ref{exa:non-exist}. Then $S$ is not a true martingale. In particular, $S$ does not satisfy $\rp$ for any $p \geq 1$.
\end{corollary}

\begin{proof}
If $S$ is a true martingale, then there is at most one solution to \eqref{non-exist-bsde} in $\sinf \times \bmo$. But we showed in Example \ref{exa:non-exist} that there is a non-zero solution to \eqref{non-exist-bsde} in $\sinf \times \bmo$.
\end{proof}

We note there are several natural generalizations of the reverse H\"older inequality to matrices, and we have essentially used the definition given in \cite{Delbaen-Tang}. In \cite{harter2019}, the authors use a slightly different definition of the reverse H\"older inequality for matrices than we have. In particular, they say that $S$ satisfies a Reverse H\"older inequality if there is a constant $C$ such that the estimate 
\begin{align*}
    \bE_{\tau}[ \sup_{\tau \leq t \leq T} |S_{\tau}^{-1} S_t|^p] \leq C
\end{align*}
holds for all stopping times $0 \leq \tau \leq T$. If this condition holds, we say that $S$ satisfies $\rphr$. Setting $\tau = 0$, we see that if $S$ satisfies $\rphr$, then $\sup_{0 \leq t \leq T} |S_t| \in \lpee$, and so $S$ is a uniformly integrable martingale. Furthermore, the computation
\begin{align*}
    \bE_{\tau}[|S_T|^p] = \bE_{\tau}[|S_{\tau} S_{\tau}^{-1} S_T|^p] \\
    \leqc \bE_{\tau} [|S_{\tau}|^p |S_{\tau}^{-1} S_T|^p] = |S_{\tau}|^p \bE_{\tau}[|S_{\tau}^{-1} S_T|^p]
\end{align*}
shows that $\rphr$ implies $\rp$ for each $p > 1$. In particular, Corollary \ref{cor:norp} also shows that if $A$ is the $\bmo$ process constructed in Example \ref{exa:non-exist}, then $S$ does not satsify $\rphr$ for any $p > 1$.

\section{Triangular BSDE}\label{sec:triangular}
\subsection{The main result}

\begin{definition}\label{def:driver} A \define{driver} is a random field $f : [0,T]
  \times \Omega \times \R^n \times (\R^d)^n \to \R^n$ such that
  \begin{enumerate}
    \item $f(\cdot,\cdot,y,\kz)$ is progressively measurable process for all
    $y,\kz$.
    \item $f(\cdot,\omega,\cdot,\cdot)$ is a continuous function for each
    $\omega$.
  \end{enumerate}
  \end{definition}

  Given a driver $f$ and a random vector $\xi$ we are interested in the
  following BSDE:
  \begin{align} \label{bsde}
  Y = \xi + \int_{\cdot}^T f(\cdot,Y,\kZ)\, dt - \int_{\cdot}^T \kZ\, d\kB
  \end{align}
  A pair $(Y,\kZ)$ consisting of a semimartingale $Y$ and a progressive process
  $\kZ$ is called a \define{solution to the BSDE} $(\xi,f)$, also denoted by
  $(Y,\kZ) \in \bsde(\xi,f)$, if $\int_0^{\cdot} \kZ\, d\kB$ is a martingale and
  \eqref{bsde} holds pathwise, a.s. If $Y\in \sinf$ and $\kZ\in \bmo$, we say
  that $(Y,\kZ)$ is a \define{$\sinf\times\bmo$-solution}, and denote that by
  $(Y,\kZ)\in \bsde_{\sinf\times\bmo}(\xi,f)$. A similar convention is used for
  other spaces of processes.

  \medskip

\begin{definition}
    
\label{def:driversmooth}
  \noindent A driver  $f$ is said to be
\begin{itemize}
\item \define{Quadratic} if there exists a constant $L$ such that for all
$y,y',\kz,\kz'$
\begin{align*}
  \abs{f(\cdot, y',\kz') - f(\cdot, y,\kz)} \leq L\abs{y'-y} +
  L \Big(1+\abs{y}+\abs{y'}+\abs{\kz}+\abs{\kz'}\Big)\abs{\kz'-\kz},
\end{align*}
and $\norm{f(\cdot,\cdot,0,0)}_{\lii}\leq L$. In case we write $f\in\qua(L)$,
with $\qua = \cup_{L} \qua(L) $.

\item \define{$C^1$-regular} if $f(t,\omega,\cdot,\cdot) \in C^1$ for all
  $t,\omega$. In this case, we write $f\in \con$.
\item \define{Malliavin-regular} if there exists a constant $L>0$ such that
  $f\in\qua(L)$ and there exists a random field $D_{\cdot}f: [0,T]^2 \times
  \Omega \times \R^d \times (\R^d)^n \to \R$ such that
    \begin{enumerate}
      \item for all $(y,\kz)$, $D_{\cdot}f(\cdot,y,\kz)$ is a version of the
      Malliavin derivative of the process $f(\cdot, y, \kz)$,
      \item $\abs{Df}\leq L$, and
      \item $\abs{ D_{\cdot}f(\cdot, y', \kz') - D_{\cdot}f(\cdot, y, \kz)} 
      \leq  L \Big( \abs{y'-y} + \abs{\kz' - \kz} \Big)$.
    \end{enumerate}
  In this case, we write $f \in \mal(L)$, with $\mal = \cup_L \mal(L)$.
  \end{itemize}
  \end{definition}
  A function $\kappa:[0,\infty) \to [0,\infty)$ is going to be called
  \define{sublinear} if it is non-decreasing, concave and $\lim_{x\to\infty}
  \tfrac{\kappa(x)}{x} = 0$. We also recall that vectors $a_1,\dots, a_M$ are
  said to \define{positively span $\R^n$} if for each $v\in \R^n$ there exist
  \emph{nonnegative} coefficients $c_1,\dots, c_M$ such that $\sum_m c_m a_m =
  v$.

  \begin{definition} A driver is  said to 
    \begin{itemize}
      \item be \define{triangular} if there exist a constant $L>0$ and a
   sublinear function $\kappa$ such that $f\in \qua(L)$ and for all $y,y'$,
   $\kz,\kz'$ and $1\leq i \leq n$, we have
  \begin{align*}
     & \abs{f^i(\cdot, y',\kz') - f^i(\cdot, y,\kz)}  \leq L \abs{y' - y} + \\
     & \qquad  + L \sum_{j = 1}^i \Big(1 + \abs{y} + \abs{y'} + \abs{\kz}
    + \abs{\kz'}\Big) \,\abs{\kz'^j - \kz^j} + 
    \\ &\qquad + L \sum_{j = i+1}^n \Big(1 + \abs{y} + \abs{y'} + 
    \kappa(\abs{\kz}) + \kappa(\abs{\kz'})\Big) \,\abs{\kz'^j - \kz^j}
  \end{align*}
  In this case, we write $f \in \tri(L, \kappa)$, with $\tri = \cup_{L,\kappa}
  \tri(L,\kappa)$.  
  \item \define{satisfy the condition (AB)} if there is a
  process $\rho \in \loi$ and a finite collection $\sam = (a_1,\dots, a_M)$ of
  vectors in $\R^n$ such that
  \begin{enumerate}
    \item $a_1,\dots, a_M$ positively span $\R^n$ 
    \item $ a_m^T f(t,\omega,y, \kz) \leq \rho + \tfrac{1}{2} \abs{a_m^T \kz}^2$
    for each $m$, for all $y,\kz$.
  \end{enumerate}
  In this case, we say that $f \in \apb(\rho, \sam)$, with
  $\apb=\cup_{\rho,\sam} \apb(\rho,\sam)$. 
\end{itemize}
\end{definition}

\begin{remark}
For a driver $f\in \qua(L)\cap \con$, we automatically have
\begin{align*}
    \abs{\pd{f}{y}(\cdot,y,\kz)}\leq L \eand \abs{\pd{f}{\kz}(\cdot,y,\kz)}\leq
    L \Big(1+ \abs{y}+\abs{\kz}\Big)
  \end{align*}
  for all $t,\omega,y$ and $\kz$. In that case, we have $f\in\tri(L,\kappa)$ if
  and only if
  \[  \abs{\pd{f^i}{\kz^j}(\cdot,y,\kz)} \leq L \Big(1+\kappa(\abs{\kz})\Big)
  \efor j>i,\] for all $t,\omega,y$ and $\kz$.
\end{remark}

\begin{theorem} \label{thm:main-tri} Assume that $\xi \in \linf$ and $f \in
 \tri(L,\kappa) \cap \mal(L) \cap \apb(\rho, \sam)$. Then \eqref{bsde} admits a
 unique solution $(Y,\kZ)\in \sibm$. This solution satisfies
 \begin{align*}
  \norm{Y}_{\sinf}+\norm*{\kZ}_{\bmo} \leq C, 
  \ewhere \const{L,\kappa,\rho,\sam, \norm{\xi}_{\linf}}.
\end{align*}
If, additionally, $\xi \in \Doi$, then $\kZ \in \lii$ and 
\[ \norm{\kZ}_{\lii} \leq C \ewhere \const{\norm{D \xi}_{L^{\infty}}, L,\kappa,\rho,\sam, \norm{\xi}_{\linf}} \] 
\end{theorem}

\subsection{Existence in the smooth case}
The rest of this section is devoted to a proof of Theorem \ref{thm:main-tri}
and we start by treating the smooth case. 

Our approach is to combine approximation, the strategy of Briand and Elie,  
our results on linear BSDE in Section \ref{sec:linear} and the following fact
which follows directly from Corollary \ref{cor:sup-lin-exc}:
\begin{lemma}
  \label{lem:tri-dif-exc}
  If $f \in \con \cap \tri(L,\kappa)$ then for each $(Y,\kZ)\in\sibm$, $\bsdeaa$ is well-posed, where
  \begin{align*}
    (\alpha, \kA) \coloneqq \big(\pd{f}{y}(\cdot, Y, \kZ), \pd{f}{\kz}(\cdot,Y,\kZ) \big).
  \end{align*} 
  Moreover, $ \opnorm{\Saa} \leq C$ where $\const{\kappa, L, \norm{Y}_{\sinf} +
  \norm{\kZ}_{\bmo}}$.
\end{lemma}

The first task is to construct an approximation scheme for the driver $f$ .
\begin{definition}
  Given a driver $f$, a sequence $(\fk)$ in $\qua$ is said to be an
  \define{approximation scheme} to $f$ if, for all $k$
  \begin{align*}
    \fk(t,\omega,\cdot,\cdot) \to f(t,\omega,\cdot,\cdot) 
    \text{ uniformly on compacts, for all } (t,\omega).
  \end{align*}
  An approximation scheme is said to be \define{stable} if $\fk(\cdot,y,\kz) =
    f(\cdot, y,\kz)$ for $\abs{y}\leq k$ and $\abs{\kz}\leq k$.
\end{definition}

\begin{proposition} \label{pro:apr-Lip} For each $f\in
  \tri(L,\kappa)\cap\con\cap\mal(L)\cap\apb(\rho,\sam)$  there exists a stable
  approximation scheme $(\fk)$ for $f$ such that
\begin{enumerate}
  \item $\fk\in \tri(L,\kappa)\cap\con\cap\mal(L)$. 
  \item for each $\xi\in\linf$ there exist $\sibm$-solutions $(\Yk, \kZk)$ of
      $\bsde(\xi,\fk)$ such that
  \begin{align}\label{apr-Lip-bnd}
    \sup_{k} \Big( \norm*{\Yk}_{\sinf} + \norm*{\kZk}_{\bmo} \Big) \leq C,\  
    \const{\norm{\xi}_{\linf}, L,\kappa,\rho,\sam}
  \end{align}
\end{enumerate}
\end{proposition}
\begin{proof}
  Let $\psi:[0,\infty) \to [0,\infty)$ be a smooth, concave and nondecreasing
  function such that $\psi(x) = x$ for $x\leq 1$ and $\psi(x) = 2$ for $x\geq 2$
  so that, in particular, $\psi'(x) \in [0,1]$ and $\psi(x)\leq x$ for all $x$.
  We define $\pik:(\R^d)^n  \to (\R^d)^n$ by 
\begin{align*}
  \pik(\kz) = \frac{k}{\abs{\kz}} \psi\Big(\frac{\abs{\kz}}{k}\Big) \kz,
\end{align*}
with the understanding that $\pik(0)=0$. 
\begin{align*}
  \fk(\cdot,y,\kz) = f(\cdot, y, \pik(\kz)).
\end{align*}
The properties of $\psi$ emphasized above imply that $\fk \in \tri(L,\kappa)
\cap \con \cap \mal(L)$. According to \cite[Theorem 4.3.1, p.84]{Zhang}, the
Lipschitz BSDE $\bsde(\xi,\fk)$ admits a unique solution $(\Yk, \kZk)$ with
$\kZk\in \lto$. 

We set  $R_t = \exp{- 2 a_m^T \Yk_t +\int_0^t 2 \rho_s\, du}$ so that
\begin{align*}
  d R = 
  R\Big( 2 a_m^T \fk(\cdot, \Yk, \kZk) - 2 \rho - 2 \abs*{a_m^T \kZk}^2\Big)\, dt 
  - 2 a_m^T \kZk\, d\kB.
  \end{align*}
Since $f \in \apb(\rho,\sam)$ and $\kZk\in \lto$ we have 
\begin{equation}\label{Ito-R}
  \begin{split}
    2 a_m^T\fk(\cdot, y, \kz) &= 2 a_m^T f(\cdot, y, \pik(\kz)) \leq 2 \rho+
     \abs{a_m^T \pik(\kz)}^2 \\ &\leq 2 \rho+
      \abs{a_m^T \kz}^2,
  \end{split}
\end{equation}
which leads to  two conclusions. 

The first one is that the process $R$ is a positive supermartingale with $R_T
\in \linf$. Consequently the process $a_m^T \Yk$ admits a uniform lower bound
which depends only on $\norm{\rho}_{\loi}$ and $\norm{\xi}_{\linf}$. Since
$\sam$ positively span $\R^n$, a uniform bound transfers to $\abs{\Yk}$ and we
can conclude that $\sup \norm{\Yk}_{\sinf}<\infty$. 

Having established the boundedness of $\Yk$, we can extract more out of
\eqref{Ito-R}. Indeed, it follows from \eqref{Ito-R} that $R_t - \tot \int_0^t
R_s \abs*{a_m^T \kZk_s}\, ds$ is a supermartingale, as well. Since $R$ is now
known to be bounded from above, we conclude that the quantity $\bE_t[ \int_t^T
\abs*{a^T_m \kZk}^2\, dt]$ admits a bound in terms of $\sup \norm*{\Yk}_{\sinf}$
only. As above, since $\sam$ positively span $\R^n$, we can conclude that $\sup
\norm*{\kZk}_{\bmo}<\infty$, and complete the proof.
\end{proof}

\begin{proposition} \label{pro:smo-tri-exist} Suppose that $\xi\in \doi$ and
  $f\in\tri(L,\kappa)\cap\con\cap\mal(L)\cap\apb(\rho,\sam)$. Then \eqref{bsde}
  admits a solution $(Y,\kZ) \in \sibm$ with
  \begin{align}\label{smo-tri-bnd}
    \norm{Y}_{\sinf}+\norm*{\kZ}_{\lii} \leq C, 
    \ewhere \const{L,\kappa,\rho,\sam, \norm{\xi}_{\linf}}.
  \end{align}
\end{proposition}
\begin{proof}
  Let $(\fk)$ be the approximation scheme for $f$ as in Proposition
  \ref{pro:apr-Lip}, and let the solutions $(\Yk,\kZk) \in \bsde(\xi,\fk)$
  satisfy \eqref{apr-Lip-bnd}. Given $\theta \in [0,T)$ we take  the
  Malliavin derivative $D_{\theta}$ of $\Yk$ (this is justified for example by
  Proposition 5.3 of \cite{karoui}) to find that on $[\theta,T]$, $D_{\theta}
  \Yk$ satisfies
    \begin{multline}\label{diff-equ}
      D_{\theta} \Yk = D_{\theta} \xi + \ict \Big( \alphak D_{\theta}\Yk +
      \kAk D_{\theta} \kZk + \betak\Big)\, dt - \ict D_{\theta} \kZk \, d\kB,
    \end{multline}
  where 
  \begin{align*}
    (\alphak, \kAk, \betak)  = 
    \Big(\pd{\fk}{y}, \pd{\fk}{\kz}, D_{\theta}f \Big)(\cdot, \Yk, \kZk)
   \end{align*}

  We interpret \eqref{diff-equ} as a linear equation for $(D_{\theta} \Yk,
  D_{\theta} \kZk)$ and note that $\bsdeaa$ is well-posed by  Lemma \ref{lem:tri-dif-exc}. Moreover, we have
  the following bound
  \begin{align*}
    \norm*{D_{\theta}\Yk}_{S^{\infty}} \leqc
    \norm{D_{\theta} \xi}_{L^{\infty}} + 
    \norm*{\betak}_{\bmoh} \leqc
  1+\norm{D_{\theta} \xi}_{L^{\infty}},
  \end{align*}
  where $\const{L,\kappa,\rho,\sam, \norm{\xi}_{\linf}}$ and where we used the
  fact that $f\in \mal(L)$ in for the last inequality. 

  Together with the identification $\kZk_{\theta} = D_{\theta} \Yk_{\theta}$
  (see Proposition 5.3 of \cite{karoui} for a precise statement) this shows that
  the sequence $\set{\norm*{\kZk}_{\lii}}$ is bounded. Hence, by the stability
  of $(\fk)$, we have 
  \begin{align*}
    f(\cdot, \Yk, \kZk) = \fk(\cdot, \Yk, \kZk), \efor k \text{ large enough.}
  \end{align*}
  Hence, if we pick a sufficiently large $k$, the pair $(Y,\kZ)\coloneqq (\Yk,
    \kZk)$  solves the original BSDE \eqref{bsde} and admits the bound
    \eqref{smo-tri-bnd}.
  \end{proof}

\begin{remark} \label{rmk:approx}
Actually, the proof of Proposition \ref{pro:smo-tri-exist} shows the following: if $\xi \in \doi$ and $f \in \tri(L,\kappa) \cap \con \cap \mal(L)$, and if there exists a stable approximation scheme $(f^{(k)})$ for $f$ such that
\begin{enumerate}
  \item $\fk\in \tri(L,\kappa)\cap\con\cap\mal(L)$. 
  \item for each $\xi\in\linf$ there exist $\sibm$-solutions $(\Yk, \kZk)$ of
      $\bsde(\xi,\fk)$ such that
  \begin{align}
    \sup_{k} \Big( \norm*{\Yk}_{\sinf} + \norm*{\kZk}_{\bmo} \Big) \leq C
  \end{align}
\end{enumerate}
then there exists a solution $(Y,\kZ) \in \sinf \times \bmo$ to \eqref{bsde} such that 
\begin{align*}
    \norm{Y}_{\sinf} + \norm{\kZ}_{\bmo} \leq C.
\end{align*}
Thus we see that the only role of the condition (AB) is to guarantee the existence of such a stable approximation scheme - if the scheme can be produced by another argument, existence of a solution is still guaranteed. 
\end{remark}

\subsection{Existence and uniqueness in the general triangular case}
Having treated the smooth case in Proposition \ref{pro:smo-tri-exist} above, we turn to
the general case. We start with two stability estimates; the first one implies
uniqueness, while then second one will be used in the existence proof below. 
\begin{proposition}  \label{pro:stab-sibm} Suppose that for $i=1,2$, $f_i \in
  \tri(L,\kappa)$, $\xi_i \in \linf$ and that $(Y_i, \kZ_i) \in \sibm$ solve
  $\bsde(f_i,\xi_i)$, respectively. Then
  \begin{multline}\label{stab-sibm-est}
    \norm{Y_2 - Y_1}_{\sinf} +
    \norm{\kZ_2 - \kZ_1}_{\bmo} \leqc
    \norm{\xi_2 - \xi_1}_{\linf}  \\ +
    \norm{f_2(\cdot, Y_1, \kZ_1) - f_1(\cdot, Y_1, \kZ_1) }_{\bmoh},
    \end{multline}
  where $\const{L, \kappa, \norm{Y_i}_{\sinf}, \norm{\kZ_i}_{\bmo}, i=1,2}$. 
  
  \medskip
  
  In particular, for $\xi\in\linf$ and $f\in \tri$, $\bsde(\xi,f)$ has at most
  one $\sinf\times \bmo$-solution.
  \end{proposition}

  \begin{proof}
  We remind the reader of the finite-difference notation introduced in Section
  \ref{sse:notation} which we apply here with $(\cdot, Y_i, \kZ_i)$ playing the
  role of $x_i=(y_i, \kz_i)$, $i=1,2$. If we set $\Delta Y = Y_2 - Y_1$ and
  $\Delta \kZ = \kZ_2 - \kZ_1$, it follows by telescoping that the processes
  $\alpha$ and $\kZ$ defined by 
  \[ \alpha = \pD{f_2}{y}  \eand \kA = \pD{f_2}{\kz}\] satisfy
  \[ f_2(\cdot,Y_2, \kZ_2) - f_2(\cdot, Y_1, \kZ_1) = \alpha \Delta Y + \kA
  \Delta \kZ.\] Therefore, with $\beta = f_2(\cdot, Y_1,\kZ_1) - f_1(\cdot,
  Y_1,\kZ_1)$, we have
  \[ \Delta Y = \ict \Big(\alpha \Delta Y + \kA \Delta \kZ + \beta\Big)\, dt -
  \ict \Delta \kZ\, d\kB.\]
  
  Thanks to the fact that $f_2\in \tri(L,\kappa)$, the process $\alpha$ is
  bounded by $L$, while 
  \begin{align} \label{est-diff}
    \abs{\kA^i_j} \leq \begin{cases}
    L\big(1+\abs{Y_1}+\abs{Y_2}+\abs{\kZ_1}+\abs{\kZ_2}\big), & j\leq i, \\
    L\big(1+\abs{Y_1}+\abs{Y_2}+\kappa(\abs{\kZ_1})
    +\kappa(\abs{\kZ_2})\big), & j>i.
  \end{cases}
  \end{align}
  It remains use Corollaries \ref{cor:sup-lin-exc} and \ref{cor:one-slic-exc}
  to conclude that the $\bsdeaa$ is well-posed with \[
  \opnorm{\Saa}\leq C,\ \const{L, \kappa, \norm{Y^1}_{\sinf},
  \norm{Y^1}_{\sinf}, \norm{\kZ^1}_{\bmo}, \norm{\kZ^2}_{\bmo}},\] which, in
  turn, implies the inequality \eqref{stab-sibm-est}. 
  \end{proof}

  \begin{proposition}  \label{pro:stab-sqlq} Suppose that for $i=1,2$, $f_i \in
    \tri(L,\kappa)$, $\xi_i \in \linf$ and that $(Y_i, \kZ_i) \in \sibm$ solve
    $\bsde(f_i,\xi_i)$, respectively. Then there exists a constant $q^* =
    q^*(L,\kappa, \norm{Y_i}_{\sinf}, \norm{\kZ_i}_{\bmo}, i=1,2)$ with the
    following property: for each $q>q^*$ we have
    \begin{multline}\label{stab-sqlq-est}
      \norm{Y_2 - Y_1}_{\sque} +
      \norm{\kZ_2 - \kZ_1}_{\ltq} \leqc
      \norm{\xi_2 - \xi_1}_{\lque} \\ +
      \norm{f_2(\cdot, Y_1, \kZ_1) - f_1(\cdot, Y_1, \kZ_1) }_{\loq},
      \end{multline}
    where $\const{q, L, \kappa, \norm{Y_i}_{\sinf}, \norm{\kZ_i}_{\bmo},
    i=1,2}$.
    \end{proposition}
    \begin{proof}
     The proof is a straightforward application of Proposition \ref{pro:lqexc}, together with the linearization technique already employed several times in this section.
    \end{proof}
  Next, we construct a smooth approximation scheme for the driver $f$.
  
  \begin{remark}
  Note that Propositions \ref{pro:stab-sibm} and \ref{pro:stab-sqlq} do no require any smoothness, only that $f_i \in \tri$. In particular, Proposition \ref{pro:stab-sibm} shows that if $f \in \tri$ and $\xi \in \linf$, then there is at most one solution $(Y,\kZ) \in \sinf \times \bmo$ to $\bsde(f,\xi)$
  \end{remark}

\begin{proposition}\label{pro:apr-smo} Given $(L,\rho,\sam)$ there exists
  $(L',\rho',\samp)$ with the following property: each $f\in \tri(L,\kappa) \cap
  \mal(L) \cap \apb(\rho,\sam)$ admits  
  an approximation scheme $(\fk)$  such that $\fk \in \tri(L',\kappa) \cap \con
  \cap \mal(L) \cap \apb(\rho', \samp)$,
\end{proposition}
Our proof below relies on the following auxillary result:
\begin{lemma} \label{lem:malliavin} Let $F : \R^m \times \Omega \to \R^n$ be a
  bounded measurable map such that
  \begin{enumerate}
    \item  for each $\omega$,  the map $x \mapsto F(x,\omega)$ is continuous and
    compactly supported. 
    \item  for each $x$, we have $\omega \mapsto F(x,\omega) \in \D$ and the map
  $(\theta, \omega, x) \mapsto D_{\theta}F(x, \omega)$ admits a measurable
  version with the following property 
  \begin{align*}
      \abs{D_{\theta} F(y, \omega ) - D_{\theta} F(x, \omega)} 
      \leq L\abs{y - x}, \text{ for some  $L\geq 0$.}
  \end{align*}
\end{enumerate} 
  Then, the random variable $\int_{\R^m} F(x, \cdot)\, dx$ is in $\D$, and 
  \begin{align*}
      D_{\theta} \int_{\R^m} F(x, \omega) dx = 
      \int_{\R^m} D_{\theta } F(x, \omega) dx, \ d\theta\times d\bP\text{-a.e.}
  \end{align*}
  
  \end{lemma}
  
  \begin{proof}
    Let $K$ be such that $f(x,\omega) = 0$ for $x \not \in [-K,K]^m$, and let
  $\set{ \sQ_j = \set{Q_{j,1},....,Q_{j,k_j}}}_{j \in \N}$ be a sequence of
  partitions of $[-K,K]^m$ by almost disjoint rectangles such that $\sQ_{j+1}$
  refines $\sQ_j$ and $\max_{i} \diam(\sQ_{j,i}) \to 0$. For each $j$ and $i
  \leq k_j$, let $x_{j,i}$ be a point in $\sQ_{j, i}$. Since $x \mapsto F(x,
  \omega)$ is continuous, we have
  \begin{align*}
      \int_{\R^m} F(x, \omega) dx = \lim_{j \to \infty} 
      \sum_{i = 1}^{k_j} \abs{Q_{j,i}} F(x_{j,i}, \omega)
  \end{align*}
  for each $\omega$. Define $X_k : \Omega \to \R^n$ by $ X_k(\omega) = \sum_{i
   =1}^{k_j} \abs{Q_{j,i}} F(x_{j,i}, \omega).$ Since $F$ is bounded, it follows
   that 
  \begin{align*}
      X_k  \xrightarrow{\ltwo} \int_{\R^m} F(x, \cdot)\, dx. 
  \end{align*}
  For each $j$, we have by linearity 
  \begin{align*}
      D X_k = 
      \sum_{i = 1}^{k_j} \abs{Q_{ji}} D F(x_{ji}, \cdot). 
  \end{align*}
  Since $x \mapsto D_{\theta} F(x, \omega)$ is $L$-Lipschitz $d\theta \times
  d\bP$ a.e., we see that 
  \begin{align*}
      DX_k \xrightarrow{L^2} \int_{\R^m} D F(x, \cdot). 
  \end{align*}
  Because $D$ is a closed operator, this completes the proof. 
  \end{proof}
  
\begin{proof}[Proof of Proposition \ref{pro:apr-smo}.]
 Let $\eta$ be a standard mollifier, i.e., a nonnegative $C^{\infty}$-function
of $(y,\kz)$ supported by $\set{\abs{y}\leq 1, \abs{\kz}\leq 1}$  such that
$\norm{\eta}_{\lone}=1$. We define $\etak(y,\kz) = k^{-(n+nd)} \eta(k y, k\kz)$
and set 
\begin{align*}
    \fk(t,\omega, \cdot, \cdot) \coloneqq f(t,\omega, \cdot, \cdot) * \etak, 
    \eforeach t,\omega, 
\end{align*}
where $*$ denotes convolution in $(y,\kz)$, so that $\fk\in\con$. Standard
properties of mollification imply that each $\fk$ is a driver in $\con$ and that
$\fk \to f$  uniformly on compacts, for each $(t,\omega)$. The estimate
\begin{multline*}
   \abs{ (\fk)^i(\cdot,y',\kz') - (\fk)^i(\cdot,y,\kz)} \leq \\
    \leq  \sup_{\abs{\delta y} \leq 1/k,
\abs{\delta\kz}\leq 1/k} \abs{ f^i(\cdot,y+\delta y,\kz+\delta\kz) -
f^i(\cdot,y'+\delta y,\kz'+\delta \kz)},
\end{multline*}
makes it easy to show that $\fk\in \tri(L',\kappa')$, where $L'$ and $\kappa'$
depend only on $L$ and $\kappa$. 

Lemma \ref{lem:malliavin} above implies the Malliavin derivative commutes with
mollification in this case and preserves boundedness and the Lipschitz property
in $(y,\kz)$. Consequently,  $\fk \in \mal(L)$.

Next, we turn to the condition $\apb$. Given $(\rho,\sam)$ such $f\in
\apb(\rho,\sam)$ we observe that for any $k$, 
\begin{align*}
    a_m^T \fk(\cdot, y, \kz) &= 
    a_m^T \big(f(\cdot, \cdot, \cdot) * \etak\big)(y,\kz) 
    \leq \sup_{|y' - y| \leq 1, |\kz' - \kz| \leq 1} 
    \frac{1}{2}a_m^T f(\cdot, y', \kz')\\
    & \leq \rho(t) + \sup_{|\kz' - \kz| \leq 1} \frac{1}{2}|a_m^T \kz'|^2 
   \leq \rho(t) + |a_m|^2 + |a_m^T \kz|^2.
\end{align*}
Thus,  $\fk\in \apb(\rho', \samp)$ where $\rho' = \rho+ \sup_m \abs{a_m}^2$ and
$\samp =  \{2 a_m\}$.
\end{proof}

\begin{proposition}
  Assume that $\xi \in \linf$ and $f \in \tri(L,\kappa) \cap \mal(L) \cap
 \apb(\rho, \sam)$. Then \eqref{bsde} admits solution $(Y,\kZ)\in \sibm$ which
 satisfies
 \begin{align*}
  \norm{Y}_{\sinf}+\norm*{\kZ}_{\bmo} \leq C, 
  \ewhere \const{L,\kappa,\rho,\sam, \norm{\xi}_{\linf}}.
\end{align*}
\end{proposition}
 \begin{proof}
Let $(\fk)$ be the approximation scheme for $f$ from Proposition
\ref{pro:apr-smo}, and let  $\{\xik\}$ be a sequence of random vectors such
that each $\xik$ has a bounded Malliavin derivative, $\sup_k \norm{\xik}_{\linf}
< \infty$ and $\xik \to \xi$ in $\lpee$ for all $p$ (see, e.g., the proof of
\cite[Theorem 2.2., p.~2931]{Briand-Elie} for a construction). 

Proposition \ref{pro:smo-tri-exist} implies that the BSDE $\bsde(\xik, \fk)$ admits a
$\sinf\times\lii$-solution $(\Yk,\kZk)$ with
$\norm{\Yk}_{\sinf}+\norm*{\kZk}_{\bmo}\leq M$, for some $M$ indepedent of $k$.

Proposition \ref{pro:stab-sqlq} guarantees the existence of a universal constant $q^*
= q^*(f,\xi)$ such that any $q>q^*$ we have
\begin{multline*}
  \norm*{\Yk - \Yl}_{\sque}+\norm*{\kZk - \kZl}_{\ltq} 
  \leqc \norm{\xik - \xil}_{\lque}  +\\
  + \norm{ \fk(\cdot, \Yk, \kZk) - \fl(\cdot, \Yk, \kZk)}_{\loq}
\end{multline*}

The estimates on $\Yk$ and $\kZk$, together with the fact that $\fk \to f$
uniformly for $(y,\kz)$ in a compact set, imply that
\begin{align*}
\fk(\cdot, \cdot, \Yk, \kZk) - \fl( \cdot, \cdot, \Yk, \kZk) \to 0 
\end{align*}
in (the product) measure as $k,l\to\infty$. Again using our a-priori estimates
on $\Yk$ and $\kZk$, we conclude that $\fk(\cdot, \cdot, \Yk, \kZk) - \fl(\cdot,
\cdot, \Yk, \kZk) \to 0$ in $\loq$ as $l,k \to \infty$. Since $\set{\xik}$ is
Cauchy in $\lque$ by construction, it follows  that $\{\Yk,\kZk\}_k$ is Cauchy
in $\sque\times\ltp$, and thus converges to $(Y,\kZ)$ in $\sque\times\ltq$ to
\eqref{bsde}. It readily follows that $(Y,\kZ)$ is a $\sque\times\ltq$-solution
of \eqref{bsde}. The estimate on $\norm{Y}_{\sinf}$ follows from the uniform
estimates on$\norm{\Yk}_{\sinf}$, and then the estimate on $\norm{\kZ}_{\bmo}$
follows from standard techniques.
\end{proof}

\begin{remark}
While approximating $f$ via convolution allows us to remove the hypothesis that $f \in \con$, removing the assumption that $f \in \mal$ is more challenging. In section 4.2 of \cite{chassagneux2021reflected}, an innovative method is introduced for approximating a driver by one which is ``discrete path-dependent". Their ideas can be used to approximate a driver $f \in \tri \cap \con$ by drivers $f^k \in 
\tri \cap \con \cap \mal$, but unless $f$ is sub-quadratic, the approximation is not strong enough to conclude that the sequence $(Y^k, \kZ^k)$ of approximate solutions is Cauchy. Therefore we leave the question of removing the assumption that $f \in \mal$ to future research. 
\end{remark}

\begin{remark}
In light of the $\sibm$ estimate in Theorem \ref{pro:smo-tri-exist}, we can strengthen
the stability results in Propositions \ref{pro:stab-sibm} in the following way: the
dependence on the full-path norm $\norm{Y_i}_{\sinf}$, $i=1,2$ in
\ref{stab-sibm-est} and \eqref{stab-sqlq-est} can be replaced by the
$\linf$-norms $\norm{\xi_i}_{\linf}$, $i=1,2$ of the terminal conditions. 
\end{remark}

\section{A triangular game} \label{sec:game}
We now describe a game in which a Nash equilibrium can be produced via our main result, Theorem \ref{thm:main-tri}. For simplicity, we set up the game in the case $d = 1$. 
Players 1 and 2 choose strategies $\alpha$, $\beta \in \bmo(\R)$.
We have a state process $X$ which evolves according to 
\begin{align*}
    X_t = B_t = \int_0^t\big(\alpha_s + \frac{\beta_s}{\sqrt{1 + |\beta_s|^2}}) ds + B_t^{(\alpha, \beta)}, 
\end{align*}
where $B^{(\alpha, \beta)} = B - \int \big(\alpha + \frac{\beta}{\sqrt{1 + |\beta|^2}}) dt$ is a Brownian motion under the probability measure $\bP^{(\alpha, \beta)}$, given by 
\begin{align*}
    \frac{d \bP^{(\alpha, \beta)}}{d\bP} = \mathcal{E}\big(\int  (\alpha + \frac{\beta_s}{\sqrt{1 + |\beta_s|^2}}) dB\big)_T.
\end{align*}
The costs are given by 
\begin{align*}
    J^1(\alpha,\beta) = \bE^{\bP^{(\alpha,\beta)}}[\int_0^T \big(h^1 + \frac{1}{2} |\alpha|^2 \big)dt + g_1(X_{\cdot})], \\
    J^2(\alpha,\beta) = \bE^{\bP^{(\alpha,\beta)}}[\int_0^T  \big(h^2 + \frac{1}{2} |\beta|^2\big) dt + g_2(X_{\cdot})], 
\end{align*}
where $g_1, g_2: C([0,T] ; \R) \to \R$ are measurable and bounded, and $h^1,h^2$ are bounded functionals which are Malliavin regular in the sense of definition \ref{def:driversmooth}. We seek a Nash equilibrium, namely a pair $(\alpha^*, \beta^*) \in \bmo(\R)$ such that 
\begin{align*}
    J_1(\alpha^*,\beta^*) \leq J_1(\alpha,\beta^*) \text{ for all } \alpha \in \bmo(\R), \\
    J_2(\alpha^*, \beta^*) \leq J_2(\alpha^*,\beta) \text{ for all } \beta \in \bmo(\R). 
\end{align*}For given $(\alpha,\beta)$, we can define 
\begin{align*}
    Y^{\alpha,\beta, 1} = \bE_t^{\bP^{\alpha}}[ g_1(X_{\cdot}) + \int_t^T \big(h^1 + \frac{1}{2} |\alpha_2|^2\big) dt], \\
    Y^{\alpha,\beta, 2} = \bE_t^{\bP^{\alpha}}[ g_2(X_{\cdot})  + \int_t^T \big(h^2 + \frac{1}{2} |\beta_2|^2\big)  dt], 
\end{align*}
then $Y^{\alpha,\beta,1}$, $Y^{\alpha,\beta,2}$ satisfy $Y_0^{\alpha,i} = J^i(\alpha,\beta)$, and also solve the BSDE
\begin{align*}
    Y^{\alpha,\beta,1} = g_1(X_T) + \int_{\cdot}^T \big(h^1 + \frac{1}{2}|\alpha|^2 + Z^{\alpha, \beta, 1}(\alpha + \frac{ \beta}{\sqrt{1 + |\beta|^2}}) \big) dt - \int_{\cdot}^T Z^{\alpha, \beta, 1} dB, \\
    Y^{\alpha,\beta,2} = g_2(X_T) + \int_{\cdot}^T \big(h^2 + \frac{1}{2}|\beta|^2 +  Z^{\alpha, \beta, 2}(\alpha + \frac{ \beta}{\sqrt{1 + |\beta|^2}}) \big) dt - \int_{\cdot}^T Z^{\alpha, \beta, 2} dB.
\end{align*}
Accordingly, we introduce the Lagrangians
\begin{align*}
    L^1(a,b, z_1) = \frac{1}{2} |a|^2 + z_1(a + \frac{ b}{\sqrt{1 + b^2}}), \\
    L^2(a,b,z_2) = \frac{1}{2} |b|^2 +  z_2(a +\frac{ b}{\sqrt{1 + b^2}}). 
\end{align*}
For fixed $z = (z_1,z_2)$ the static game with payoffs $L^1(\cdot, \cdot, z_1)$, $L^2(\cdot, \cdot, z_2)$ has a unique Nash equilibrium given by 
\begin{align*}
    a^*(z) = - z_1, \\
    b^*(z) = \phi^{-1}(-z_2), 
\end{align*}
where $\phi : \R \to \R$ is given by $\phi(b) = b(1 + b^2)^{3/2}$. Setting $L^i(z) = L^i(a^*(z), b^*(z), z)$, we have 
\begin{align*}
    L^1(z) = -\frac{1}{2} |z_1|^2 + z_1 \frac{ \phi^{-1}(-z_2)}{\sqrt{ 1 +  |\phi^{-1}(-z_2)|^2}}, \\
    L^2(z) = \frac{1}{2} |\phi^{-1}(-z_2)|^2 - z_2 z_1 + z_2\frac{ \phi^{-1}(-z_2)}{\sqrt{ 1 +  |\phi^{-1}(-z_2)|^2}}
\end{align*}
So, we pose the BSDE
\begin{align} \label{gamebsde}
    Y_{\cdot} = g(X_{\cdot}) + \int_{\cdot}^T f(\cdot, Z) dt - \int_{\cdot}^T Z dB, 
\end{align} where $f(t,\omega, z) = (L^1(z) + h^1_t(\omega), L^2(z) + h^2_t(\omega))$. 

\begin{proposition} \label{prop:game}
  The BSDE \eqref{gamebsde} has a unique solution $(Y,\kZ) \in \sinf \times \bmo$, and the pair $(a^*(Z), b^*(Z))$ is a Nash equilibrium.
\end{proposition}

\begin{proof}
First, suppose that $\xi = g(B) \in \doi$. It is straightforward to check that $f \in \tri(L,\kappa)$ for some $L$ and $\kappa$. Define $f^{(k)}(t,\omega,z) = f(t,\omega,\pi^{(k)}(z))$, where $\pi^{(k)}$ is as in the proof of Proposition \ref{pro:apr-Lip}. Then $f^{(k)} \in \tri(L,\kappa) \cap C$. Let $(Y^{(k)}, Z^{(k)})$ be the unique solution to the Lipschitz BSDE 
\begin{align*}
Y_{\cdot}^{(k)} = \xi +  \int_{\cdot}^T f^{(k)}(\cdot,Z^{(k)}) dt - \int_{\cdot}^T Z^{(k)} dB. 
\end{align*}
Recall that $\psi^{(k)}(z) = \frac{k}{|z|} \psi(\frac{|z|}{k}) z$, where $\psi$ is as in the proof of Proposition \ref{pro:apr-Lip}. In particular, if we define $g^{(k)}(z) = \psi(\frac{|z|}{k})$, and also define $l(z_2) = \frac{\phi^{-1}(-z_2)}{\sqrt{1 + \phi^{-1}(-z_2)}}$, we can write  
\begin{align*}
    Y^{(k),1}_{\cdot}
    = \xi + \int_{\cdot}^T Z^{(k),1} \cdot A \, dt- \int_{\cdot}^T Z^{(k), 1} dB 
    = \xi + \int_t^T Z^{(k),2} dB^{(k)}, 
\end{align*}
where 
\begin{align*}
    A = -\frac{1}{2} |g^{(k)}(Z^{(k)})|^2 Z^{(k),1} + g^{(k)}(Z^{(k)}) l(\pi^{(k),2}(Z^{(k)})), \, \,
    B^{(k)} = B - \int A dt. 
\end{align*}
In particular, we see that $Y^{(k),1}$ is a martingale under an equivalent probability measure. The same argument works for $Y^{(k),2}$, leading to the estimate
\begin{align*}
    \sup_k \norm{Y^{(k)}}_{\sinf} \leqc \norm{\xi}_{\linf} + \norm{h}_{\sinf} 
\end{align*}
Since $|f^{(k),1}(z)| \leq |z_1|^2 + 1 + |h^1|$, a corresponding estimate on $Z^{(k),1}$ follows from studying the dynamics of $\exp{\lambda Y_t^{(k),1}}$ for large $\lambda$, as in the scalar case. Estimates on $Z^{(k),2}$ can then be obtained in a similar way, by considering the process $\exp{\lambda Y_t^{(k),2}} + |Z_t^{(k),1}|^2$. Thus we arrive at the estimate
\begin{align*}
    \sup_k \big(\norm{Y^{(k)}}_{\sinf} + \norm{ Z^{(k)}}_{\bmo}\big) \leq C, \, \, \const{\norm{\xi}_{\linf}, \norm{h}_{\sinf}}. 
\end{align*}
Thus by Proposition \ref{pro:smo-tri-exist} (actually a slight generalization, see remark \ref{rmk:approx}), we have a solution $(Y,Z) \in \sinf \times \bmo$ to \eqref{gamebsde}, which satisfies
\begin{align*}
    \norm{Y}_{\sinf} + \norm{Z}_{\bmo} \leq C, \, \, \const{\norm{\xi}_{\loi}}. 
\end{align*}
The same approximation argument appearing in the proof of Theorem \ref{thm:main-tri} now allows us to remove the assumption that $\xi \in \doi$. Uniqueness is given by Corollary \ref{pro:stab-sibm}. Finally, that $(\alpha^*,\beta^*) = (a^*(Z), b^*(Z))$ is a Nash equilibrium follows as in the proof of Proposition 3.6 of \cite{xing2018}. This completes the proof. 
\end{proof}

\bibliographystyle{amsalpha}
\bibliography{qtdrivers}
\end{document}